\documentclass[10pt,letterpaper]{article}
\usepackage[letterpaper, top=.95in, bottom=1in, left=1.2in, right=1.2in, marginparwidth=2.3cm, marginparsep=1mm, includefoot]{geometry}

\usepackage{latexsym,
amsthm,
color,
amssymb,
url,
bm,
cite,
amssymb,
microtype,
amsmath,
amsfonts,
mathtools}

\usepackage{amssymb}
\usepackage{microtype}
\usepackage{amsmath}
\usepackage{amsfonts}
\usepackage{mathtools}
\usepackage[mathscr]{euscript}
\usepackage{amsthm}
\usepackage{nicefrac}
\usepackage{dsfont}

\usepackage{caption}
\usepackage{setspace}
\usepackage{subcaption}
\usepackage{setspace}

 \usepackage[pdftex, plainpages = false, pdfpagelabels, 
                 bookmarks = true,
                 bookmarksopen = true,
                 bookmarksnumbered = true,
                 breaklinks = true,
                 linktocpage,
                 pagebackref,
                 colorlinks = true,  
                 linkcolor = blue,
                 urlcolor  = blue,
                 citecolor = red,
                 anchorcolor = green,
                 hyperindex = true,
                 hyperfigures
                 ]{hyperref}

\usepackage[inline]{enumitem}
\usepackage[mathscr]{euscript}
\usepackage{amsthm}
\usepackage{nicefrac}
\usepackage{dsfont}
\usepackage{tikz}
\usepackage[ruled,linesnumbered]{algorithm2e}

\usepackage{nicefrac}
\usepackage[mathscr]{euscript}

\RequirePackage[T1]{fontenc}
\RequirePackage[utf8]{inputenc}
\usepackage[english]{babel}
\usepackage{alphabeta}
\usepackage{graphicx}

\usetikzlibrary{math}

\makeatletter
\input{colordvi}

\usepackage{aliascnt}

\hypersetup{
    colorlinks,
    linkcolor={red!75!black},
    citecolor={green!75!black},
    urlcolor={blue!75!black},
}

\definecolor{denim}{rgb}{0.08, 0.38, 0.74}

\definecolor{denim}{rgb}{0.08, 0.38, 0.74}
\definecolor{BrilliantRose}{rgb}{1.0, 0.33, 0.64}
\definecolor{amethyst}{rgb}{0.6, 0.4, 0.8}

\newcommand{\cref}[1]{\autoref{#1}}

\usepackage{todonotes}

\newcommand{\rred}[1]{#1}
\newcommand{\revone}[1]{}

\newcommand{\revtwo}[1]{}

\newcommand{\revthree}[1]{}

\usepackage[normalem]{ulem}
\usepackage{xcolor}
\newcommand{\stra}[1]{%
  {\color{red}\sout{\textcolor{black}{#1}}}%
}

\newcommand{\change}[2]{#2}

\newcommand{\sed}[1]{}%
\newcommand{\quest}[1]{}%
\newcommand{\ans}[1]{}%
\newcommand{\sw}[1]{}%

\newcommand{\remove}[1]{}

\newcounter{func}
\newcommand{\funref}[1]{\hyperref[#1]{f_{\ref*{#1}}}} 

\tikzset{black node/.style={draw, circle, fill = black, minimum size = 5pt, inner sep = 0pt}}
\tikzset{white node/.style={draw, circlternary_treese, fill = white, minimum size = 5pt, inner sep = 0pt}}
\tikzset{normal/.style = {draw=none, fill = none}}
\tikzset{lean/.style = {draw=none, rectangle, fill = none, minimum size = 0pt, inner sep = 0pt}}
\usetikzlibrary{decorations.pathreplacing}
\usetikzlibrary{arrows.meta}

\usetikzlibrary{shapes}
\tikzset{diam/.style={draw, diamond, fill = black, minimum size = 7pt, inner sep = 0pt}}

\newcommand{\Bcal}{\mathcal{B}}

\newcommand{\Ocal}{\mathcal{O}}
\newcommand{\Pcal}{\mathcal{P}}

\newcommand{\Nbbb}{\mathbb{N}}

\newcommand{\Sbbb}{\mathbb{S}}

\RequirePackage{stmaryrd}
\usepackage{textcomp}
\DeclareUnicodeCharacter{2286}{\subseteq}
\DeclareUnicodeCharacter{2192}{\ifmmode\to\else\textrightarrow\fi}
\DeclareUnicodeCharacter{2203}{\ensuremath\exists}
\DeclareUnicodeCharacter{183}{\cdot}
\DeclareUnicodeCharacter{2200}{\forall}
\DeclareUnicodeCharacter{2264}{\leq}
\DeclareUnicodeCharacter{2265}{\geq}
\DeclareUnicodeCharacter{8614}{\mathbin{\mapsto}}
\DeclareUnicodeCharacter{8656}{\Leftarrow}
\DeclareUnicodeCharacter{8657}{\Uparrow}
\DeclareUnicodeCharacter{8658}{\Rightarrow}
\DeclareUnicodeCharacter{8659}{\Downarrow}
\DeclareUnicodeCharacter{8669}{\rightsquigarrow}
\newcommand{\eqdef}{\stackrel{{\scriptsize\rm def}}{=}}
\DeclareUnicodeCharacter{8797}{\eqdef}
\DeclareUnicodeCharacter{8870}{\vdash}
\DeclareUnicodeCharacter{8873}{\Vdash}
\DeclareUnicodeCharacter{22A7}{\models}
\DeclareUnicodeCharacter{9121}{\lceil}
\DeclareUnicodeCharacter{9123}{\lfloor}
\DeclareUnicodeCharacter{9124}{\rceil}
\DeclareUnicodeCharacter{2208}{\in}
\DeclareUnicodeCharacter{9126}{\rfloor}
\DeclareUnicodeCharacter{9655}{\triangleright}
\DeclareUnicodeCharacter{9665}{\triangleleft}
\DeclareUnicodeCharacter{9671}{\diamond}
\DeclareUnicodeCharacter{9675}{\circ}
\DeclareUnicodeCharacter{10178}{\bot}
\DeclareUnicodeCharacter{10214}{} 
\DeclareUnicodeCharacter{10215}{} 
\DeclareUnicodeCharacter{10229}{\longleftarrow}
\DeclareUnicodeCharacter{10230}{\longrightarrow}
\DeclareUnicodeCharacter{10231}{\longleftrightarrow}
\DeclareUnicodeCharacter{10232}{\Longleftarrow}
\DeclareUnicodeCharacter{10233}{\Longrightarrow}
\DeclareUnicodeCharacter{10234}{\Longleftrightarrow}
\DeclareUnicodeCharacter{10236}{\longmapsto}
\DeclareUnicodeCharacter{10238}{\Longmapsto} 
\DeclareUnicodeCharacter{10503}{\Mapsto}    
\DeclareUnicodeCharacter{10971}{\mathrel{\not\hspace{-0.2em}\cap}}
\DeclareUnicodeCharacter{65294}{\ldotp}
\DeclareUnicodeCharacter{65372}{\mid}

\definecolor{DarkTangerine}{rgb}{1.0, 0.66, 0.07}
\definecolor{darkyellow}{rgb}{.7, .6, 0.0}
\definecolor{CornflowerBlue}{rgb}{0.39, 0.58, 0.93}
\definecolor{DarkGoldenrod}{rgb}{0.72, 0.53, 0.04}
\definecolor{BritishRacingGreen}{rgb}{0.0, 0.26, 0.15}
\definecolor{AO}{rgb}{0.0, 0.5, 0.0}
\definecolor{MidnightBlack}{rgb}{0.1,0.1,.34}
\definecolor{MidnightBlue}{rgb}{0.1,0.1,0.43}
\definecolor{Black}{rgb}{0,0, 0}
\definecolor{Blue}{rgb}{0, 0 ,1}
\definecolor{Red}{rgb}{1, 0 ,0}
\definecolor{White}{rgb}{1, 1, 1}
\definecolor{DeepMagenta}{rgb}{0.8, 0.0, 0.8}
\definecolor{grey}{rgb}{.6, .6, .6}
\definecolor{darkgrey}{rgb}{.33, .33, .33}
\definecolor{Mygreen}{rgb}{.0, .7, .0}
\definecolor{Yellow}{rgb}{.55,.55,0}
\definecolor{Mustard}{rgb}{1.0, 0.86, 0.35}
\definecolor{applegreen}{rgb}{0.55, 0.71, 0.0}
\definecolor{darkturquoise}{rgb}{0.0, 0.81, 0.82}
\definecolor{celestialblue}{rgb}{0.29, 0.59, 0.82}
\definecolor{green_yellow}{rgb}{0.68, 1.0, 0.18}
\definecolor{crimsonglory}{rgb}{0.75, 0.0, 0.2}
\definecolor{darkmagenta}{rgb}{0.30, 0.0, 0.30}
\definecolor{magenta}{rgb}{0.50, 0.0, 0.50}
\definecolor{internationalorange}{rgb}{1.0, 0.31, 0.0}
\definecolor{darkorange}{rgb}{1.0, 0.55, 0.0}
\definecolor{ao}{rgb}{0.0, 0.5, 0.0}
\definecolor{awesome}{rgb}{1.0, 0.13, 0.32}
\definecolor{darkcyan}{rgb}{0.0, 0.50, 0.50}
\definecolor{violet}{rgb}{0.93, 0.51, 0.93}
\definecolor{brown}{rgb}{0.65, 0.16, 0.16}
\definecolor{orange}{rgb}{1.0, 0.65, 0.0}
\definecolor{DarkGreen}{rgb}{0,.5,0}
\definecolor{BostonUniversityRed}{rgb}{0.8, 0.0, 0.0}
\definecolor{BrightLavender}{rgb}{0.75, 0.58, 0.89}
\definecolor{DarkLavender}{rgb}{0.45, 0.31, 0.59}
\definecolor{ChromeYellow}{rgb}{1.0, 0.65, 0.0}

\definecolor{DarkGreen}{rgb}{0.547,.5,0}

\newcommand{\midnightblack}[1]{{\color{MidnightBlack}#1}}

\definecolor{Red}{rgb}{1, 0 ,0}
\definecolor{Blue}{rgb}{0, 0 ,1}

\newtheorem{theorem}{Theorem}[section]

\newaliascnt{question}{theorem}

\aliascntresetthe{question}

\newaliascnt{lemma}{theorem}
\newtheorem{lemma}[lemma]{Lemma}
\aliascntresetthe{lemma}

\newaliascnt{claim}{theorem}
\newtheorem{claim}[claim]{Claim}
\aliascntresetthe{claim}

\newaliascnt{invariant}{theorem}

\aliascntresetthe{invariant}

\newaliascnt{proposition}{theorem}
\newtheorem{proposition}[proposition]{Proposition}
\aliascntresetthe{proposition}

\newaliascnt{observation}{theorem}
\newtheorem{observation}[observation]{Observation}
\aliascntresetthe{observation}

\newaliascnt{corollary}{theorem}
\newtheorem{corollary}[corollary]{Corollary}
\aliascntresetthe{corollary}

\newaliascnt{definition}{theorem}
\newtheorem{definition}[definition]{Definition}
\aliascntresetthe{definition}

\newaliascnt{conjecture}{theorem}
\newtheorem{conjecture}[conjecture]{Conjecture}
\aliascntresetthe{conjecture}

\newaliascnt{remark}{theorem}
\newtheorem{remark}[remark]{Remark}
\aliascntresetthe{remark}

\newaliascnt{counterexample}{theorem}

\aliascntresetthe{counterexample}

\newcommand{\hh}{\end{document}}

\newcommand{\manman}[1]{\textsf{#1}}
\newcommand{\tw}{{\sf tw}\xspace}

\newcommand{\cupall}{\pmb{\pmb{\bigcup}}}

\newcommand{\poly}{\text{$\mathsf{poly}$}\xspace}

\newcommand{\bw}{{\sf bw}\xspace} 
 
\usepackage{color}

\title{Approximating branchwidth on parametric extensions\\ of planarity\thanks{Emails of authors: 
  \manman{sedthilk@thilikos.info}, \manman{wiederrecht@kaist.ac.kr}}}

\author{\bigskip\large Dimitrios M. Thilikos\thanks{LIRMM, Univ Montpellier, CNRS, Montpellier, France.}~$^{,}$\thanks{Supported by the French-German Collaboration ANR/DFG Project UTMA (ANR-20-CE92-0027), the ANR project GODASse ANR-24-CE48-4377,  and by the Franco-Norwegian project PHC AURORA 2024 (Projet n°\! 51260WL).}\and\large 
 \and\large 
 Sebastian Wiederrecht\thanks{School of Computing, KAIST, Daejeon, South Korea.}}
\date{}

\begin{document}

\maketitle

\begin{abstract}
\noindent The \textsl{branchwidth} of a graph has been introduced by Roberson and Seymour as a measure of the tree-decomposability of a graph, alternative to treewidth.
Branchwidth is polynomially computable on planar graphs by the celebrated ``Ratcatcher'' algorithm of Seymour and Thomas. We explore how this algorithm can be extended to minor-closed graph classes beyond planar graphs, as follows:
Let $H_{1}$ be a graph embeddable in the torus and $H_{2}$ be a graph embeddable in the projective plane.
We prove that every $\{H_{1},H_{2}\}$-minor free graph $G$ contains a subgraph $G'$ whose branchwidth differs from that of $G$ by a constant depending only on $H_1$ and $H_2$. Moreover, the graph $G'$ admits a tree decomposition 
where all torsos are planar.
This decomposition allows for a constant-additive approximation of branchwidth:
For $\{H_{1},H_{2}\}$-minor free graphs, there is a constant $c$ (depending on $H_{1}$ and $H_{2}$) and an $\Ocal(|V(G)|^{3})$-time algorithm that, given a graph $G$, outputs a value $b$ such that the branchwidth of $G$ is between $b$ and $b+c$.
\end{abstract}

\section{Introduction}

Graph parameters are important in graph algorithm design as they abstract away certain structural characteristics of graphs that  appear to be useful in  certain algorithmic applications.
Perhaps the most eminent among all known graph parameters is \textsl{treewidth} as it serves as a measure of the tree-decomposability of a graph. 
Graphs with small treewidth are those admitting a \textsl{tree decomposition} of small width, i.e.,  a tree-like arrangement into bounded size pieces (see \cref{def_treewidth}).
To have such a decomposition at hand is important as, for many problems on graphs, it enables the design of dynamic programming algorithms \cite{Courcelle90them}. 

While a considerable amount of research has been concentrated to treewidth, one may not neglect the attention that has been given to alternative tree decomposability measures.
The most famous ``cousin'' of  treewidth is \textsl{branchwidth}, defined in terms of \emph{branch decompositions}, that provide an alternative way to see graphs as tree-like structures. 
Branchwidth was introduced by Robertson \rred{and} Seymour (and was adopted henceforth) in the tenth volume of their Graph Minors series~\cite{robertson1991graph}.
One of the reasons for showing preference to branchwidth is that it enjoys combinatorial properties that makes it easier to handle and generalize than treewidth.
Fundamental attempts to generalize the theory of graph minors to matroids used branch decompositions~\cite{GeelenGW02branc,GeelenW02branc}.
Also, several algorithmic and meta-algorithmic results related to dynamic programming procedures in matroids and hypergraphs have been stated in terms of branch decompositions~\cite{HlinenyO07findi,Hlineny06branc,Hlineny05apara,JeongKO21findi}. 

For more practical purposes, branchwidth appeared to be useful for particular implementations.
Cook and Seymour  used branch decompositions to solve the ring routing problem, related to the design of reliable cost effective SONET networks and for solving TSP~\cite{CookS93}.
Alekhnovich \rred{and}%
\revthree{1. page 2, line 10: why not ``and''  instead of ``\&'' ?} 
\ans{Fixed.}Razborov~\cite{AlekhnovichR11satis} used
branchwidth of hypergraphs to design algorithms for SAT.
Finally, Dorn and Telle introduced decompositions that 
combined algorithm design on tree and branch decompositions to a  dynamic programming scheme~\cite{DornT09semin}.

\vspace{-0mm}\subsection{Algorithms for branchwidth}

As observed by Robertson and Seymour in~\cite{robertson1991graph}, treewidth and branchwidth 
are linearly equivalent, in the sense that the one is a constant factor approximation of the other.
They are both \textsf{NP}-complete to compute~\cite{SeymourT94callr,ArnborgCP93compl} and they can both be computed by algorithms that are fixed parameter tractable (\textsf{FPT}), when parameterized by their value~\cite{BodlaenderK96effic,Bodlaender96aline,BodlaenderT97const}.
Interestingly, the two parameters may behave quite differently when it comes to their computation in special graph classes.
For instance, Kloks, Kratochvíl, and Müller proved in~\cite{KloksKM05compu} that computing branchwidth is {\sf NP}-complete for split graphs where treewidth is polynomially computable (see also~\cite{PaulT09branc}).

\paragraph{Computing branchwidth of planar graphs.}
Perhaps the most important algorithmic result on branchwidth is the one of Seymour and Thomas~\cite{SeymourT94callr}, stating that it can be computed in time $\Ocal(|V(G)|^3)$ on planar graphs.
This is due to the celebrated Ratcatcher algorithm in~\cite{SeymourT94callr} that is actually working and explained in the more general context of planar hypergraphs.
In~\cite{GuT08optim}, Gu \change{\&}{and} Tamaki
\revthree{2. page 2, line 32: Gu is also an author of [24]} 
\ans{Fixed.}
gave a time $\Ocal(|V(G)|^3)$ algorithm that outputs an optimal branch decomposition of a planar graph.
Later Gu and Tamaki gave constant-factor approximations of branch decompositions of planar graphs, running in time $\Ocal(n^{1+\epsilon}\cdot \log n)$~\cite{GuTT11const}.
A first step towards extending the applicability of the Ratcatcher algorithm further than planar graphs was done by 
Demaine,  Hajiaghayi,  Nishimura, Ragde, and Thilikos in~\cite{DemaineHNRT04appro}.
According to the results in~\cite{DemaineHNRT04appro}, if $H$ is a singly-crossing minor%
\revthree{3. page 2, line 40: what does it mean ``single-crossing minor'' ?}%
\ans{A definition is given in footnote number 1.}
\footnote{A graph $H$ is a \emph{minor} of a graph $G$ if $H$ can be obtained from some subgraph of $G$ by contracting edges.
A graph is \emph{singly-crossing} if it can be embedded in the plane with at most one crossing. Moreover, a graph is a \emph{singly-crossing minor} if it is a minor of some singly-crossing graph.}, then branchwidth can be computed in time $\Ocal(|V(G)|^3)$ on $H$-minor free graphs.\footnote{While~\cite{DemaineHNRT04appro} claims a constant-factor approximations for both treewidth and branchwidth, an exact algorithm on singly-crossing minor-free graphs for branchwidth appears implicitly in the proofs of~\cite{DemaineHNRT04appro}.}
\rred{Single-crossing minor free graphs are a natural way to extend planarity, as $K_{3,3}$ and $K_{5}$ are already singly-crossing
and, as observed by in \cite{RobertsonS93Excluding}, they can be constructed by gluing 
together planar graphs and bounded size graphs via $k$-clique sums, for $k≤3$.}
To our knowledge, no polynomial algorithm for branchwidth is known for minor-closed graph classes \rred{that can be seen as structural extensions of planarity beyond}%
\revthree{4.  page 2, line 40 (also in the abstract): it is not very clear to me why we go from planar graphs to graphs excluding something slightly planar as a minor: maybe this can be somehow better motivated.}%
\ans{We added some text for this.} 
singly-crossing minor-free graphs.

\subparagraph{Approximation algorithms of  branchwidth.}
As branchwidth and treewidth are linearly equivalent, every approximation algorithm for the one is also an approximation algorithm for the other. 
As a first step, Bodlaender,  Gilbert,  Hafsteinsson, and Kloks proved in~\cite{BodlaenderGHK95appro} that  both parameters admit a polynomial time $\Ocal(\log n)$-approximation algorithm.
Moreover, according to Feige, Hajiaghayi, and Lee, both treewidth and branchwidth admit a (randomized) polynomial time $\Ocal(\sqrt{\log n})$-approximation algorithm~\cite{FeigeHL08impro}.
Also, this can be improved to a $\Ocal(h^2)$-approximation for graphs excluding some graph on $h$ vertices as a minor~\cite{FeigeHL08impro}.
This last result implies that branchwidth admits a polynomial time constant-factor approximation in every minor-closed graph class.
The emerging question is:

\vspace{-1mm}
\begin{quotation}
\noindent \textsl{Are there minor-closed graph classes where a ``better than constant-factor'' approximation can be achieved?}
\end{quotation}
\vspace{-1mm}
This question is the main algorithmic motivation of this paper. 

\vspace{-0mm}\subsection{Approximating branchwidth}

Let $\mathcal{E}_{1},$ and $\mathcal{E}_{2}$ be  the classes of graphs that are embeddable in the torus and the projective plane   respectively. 
Our main result is that whenever we exclude some graph in $\mathcal{E}_{1}$ and some graph in $\mathcal{E}_{2}$ as minors, a \textsl{constant-additive} approximation of branchwidth is possible.
This extends the result of \cite{DemaineHNRT04appro} on singly-crossing minor free graphs.

\begin{theorem}\label{arst_cpe}
Branchwidth admits a polynomial-time additive approximation algorithm for every graph class that excludes as a minor a graph    $H_{1}\in \mathcal{E}_{1}$ and a graph $H_{2}\in \mathcal{E}_{2}$. 
In particular, there is a function $f_{\ref{arst_cpe}}\colon\mathbb{N}\to\mathbb{N}$ and an algorithm that, given as input two graphs  $H_{1}\in \mathcal{E}_{1}$ and $H_{2}\in \mathcal{E}_{2},$ and an $\{H_{1},H_{2}\}$-minor free graph $G,$ outputs a value $b$ where $b\leq \bw(G)\leq  b+f_{\ref{arst_cpe}}(k),$ where $k=|V(H_{1})|+|V(H_{2})|$. Moreover, $f_{\ref{arst_cpe}}(k)=2^{\poly(k)}$ and the algorithm runs  in time $\Ocal(2^{2^{\poly(k)}}|V(G)|^{3})$.
\end{theorem}

\noindent 
In the above, by the term \emph{$\{H_{1},H_{2}\}$-minor free graph} we mean a graph that does neither contain $H_{1}$ nor $H_{2}$ as a minor.
We also use the term $\poly(x)$ as a shortcut of $x^{O(1)}$.
We next remark that \cref{arst_cpe} cannot be extended to \change{general}{all}  graphs,%
\revone{1. Page 3, line 29: ``General graphs''  - do you mean to all graphs?}%
\ans{Fixed.} 
assuming $\mathsf{P}\neq\mathsf{NP}$.

\begin{theorem}\label{lower}
Branchwidth does not admit a polynomial time additive approximation algorithm, unless $\mathsf{P}=\mathsf{NP}$.
\end{theorem}

An easy consequence of \cref{arst_cpe} is the following EPTAS.

\begin{theorem}\label{mainth}
Branchwidth admits an EPTAS for every graph class that excludes as a minor a graph $H_{1}\in \mathcal{E}_{1}$ and a graph $H_{2}\in \mathcal{E}_{2}$.
In particular, there is a function $f_{\ref{mainth}}\colon\mathbb{N}\to\mathbb{N}$ and an algorithm that, given as input two graphs  $H_{1}\in \mathcal{E}_{1}$ and $H_{2}\in \mathcal{E}_{2},$ and an $\{H_{1},H_{2}\}$-minor free graph $G,$ outputs a value $b$ where $b\leq \bw(G) \leq (1+\epsilon)b$. Moreover, the algorithm runs in time  $\Ocal(2^{2^{\poly(k)}}|V(G)|^{3}+f_{\ref{mainth}}(\epsilon,k)|V(G)|)$, where  $k=|V(H_{1})|+|V(H_{2})|$ and $f_{\ref{mainth}}(\epsilon,k)=2^{\poly(\frac{1}{\epsilon}\cdot 2^{\poly(k)})}$.
\end{theorem}

\noindent We stress that, in Theorems \ref{arst_cpe} and \ref{mainth}, the graphs $H_{1}$ and $H_{2}$ might be the same graph, i.e., a graph that can be embedded both in the projective plane and the torus. For instance, this implies 
that the algorithms of \cref{arst_cpe} and \cref{mainth} apply    for $K_{6}$-minor free graphs or, alternatively, for graphs excluding the Petersen graph (or any other graph in the Petersen family) as a minor.

\vspace{-0mm}\subsection{Main combinatorial result}
\label{subsec_main_c}

We define the \emph{elementary annulus wall}
\revone{2. Page 3, line 54: Elementary annulus wall should be defined properly. This description is fine for an informal introduction, but not rigorous enough. It should also be emphasized that the vertices are of degree 3.
}
\ans{We explain the maximum degree 3 outcome of the matching removal. 
We also fully detail  the definitions in {\cref{subsec_par_gr}}.}

 as the    parametric graph $\mathscr{W}^{\manman{A}}=\langle \mathscr{W}_{k}^{\manman{A}}\rangle_{k\in \mathbb{N}}$, where  $\mathscr{W}_{k}^{\manman{A}}$ is the graph obtained from a $(k\times 8k)$-cylindrical grid
 $\mathscr{G}_{k}$ by removing some perfect matching that does not intersect the ``cycles'' of $\mathscr{G}_{k}$. \rred{Notice that the removal of the perfect matching leaves a graph of maximum degree $3$.}
We  consider two more parametric graphs, namely the \emph{elementary handle wall}  $\mathscr{W}^{\manman{H}}=\langle \mathscr{W}_{k}^{\manman{H}}\rangle_{k\in \mathbb{N}}$ and the \emph{elementary crosscap wall} $\mathscr{W}^{\manman{C}}=\langle \mathscr{W}_{k}^{\manman{C}}\rangle_{k\in \mathbb{N}},$   where  $\mathscr{W}_{k}^{\manman{H}}$ and $\mathscr{W}_{k}^{\manman{C}}$ are obtained by adding edges inside the ``internal disk'' of the elementary annulus wall  $\mathscr{W}_{k}^{\manman{A}}$, as indicated in \cref{glakeopgrid} \rred{(see \cref{subsec_par_gr} for the complete definitions)}.

We may see graphs in $\mathcal{E}_{1}$ and $\mathcal{E}_{2}$ as minors of the \emph{elementary handle wall}  $\mathscr{W}^{\manman{H}}=\langle \mathscr{W}_{k}^{\manman{H}}\rangle_{k\in \mathbb{N}}$ and the \emph{elementary crosscap wall} $\mathscr{W}^{\manman{C}}=\langle \mathscr{W}_{k}^{\manman{C}}\rangle_{k\in \mathbb{N}}.$    Indeed, as proved by Gavoille and Hilaire in \cite[\rred{Theorem 1}]{gavoille2023minor}
\revone{3. Page 4, line 18: This ``observation''  should be cited as a theorem of [18].}%
\ans{We now mention that this result is  \cite[Theorem 1]{gavoille2023minor}. Also we mention that \cref{kessent_exl} is a restatement of this result.}
, there is a constant $c$ such that 
every graph $H$ in $\mathcal{E}_{1}$ (resp. $\mathcal{E}_{2}$)
 is a minor of $\mathscr{W}_{k}^{\manman{C}}$ (resp. $\mathscr{W}_{k}^{\manman{H}}$), for every  $k\geq c\cdot |V(H)|^2$. 

A \emph{handle wall} (resp. \emph{crosscap wall}) of \emph{order} $k$
is some subdivision of $\mathscr{W}_{k}^{\manman{H}}$ (resp. $\mathscr{W}_{k}^{\manman{C}}$). A \emph{handle wall} (resp. \emph{crosscap wall}) of a graph $G$ is a subgraph of $G$ that is a \emph{handle wall} (resp. \emph{crosscap wall}). It is easy to observe that the exclusion of $\mathscr{W}_{k}^{\manman{H}}$ (resp. $\mathscr{W}_{k}^{\manman{C}}$) as  \rred{a} minor is equivalent to the exclusion of a \emph{handle wall} (resp. \emph{crosscap wall}) of order $k$ as a subgraph.
This is because the maximum degree of $\mathscr{W}_{k}^{\manman{H}}$ and $\mathscr{W}_{k}^{\manman{C}}$ is at most three for all $k$.
This, combined with the aforementioned {results of} \cite{gavoille2023minor}, \change{implies 
the following}{permits to restate them as follows}:

\begin{observation}
\label{kessent_exl}
{There is some constant $c$ such that if a graph excludes some graph $H$ in $\mathcal{E}_{1}$ (resp. $\mathcal{E}_{2}$) as a minor, it also excludes the handle (resp. crosscap) wall  of order $c\cdot |V(H)|^{2}$ as a subgraph.}
\end{observation}

\begin{figure}[h]
\begin{center}
\includegraphics[height = 4.5cm]{elementary_crosscup_handle_grid_wall}
\end{center}
\caption{The {elementary annulus wall}  $\mathscr{W}_{9}^{\manman{A}}$, the {elementary handle wall} $\mathscr{W}_{9}^{\manman{H}}$, and   the {elementary crosscap wall} $\mathscr{W}_{9}^{\manman{C}}$ \rred{(no vertex exists in the common intersection of red and blue edges of  $\mathscr{W}_{9}^{\manman{C}}$)}.
}\label{glakeopgrid}
\end{figure}
\revthree{5. page 4, Figure 1: maybe it can be underlined that where the blue and red edges cross there is no vertices.}
\ans{Done.}

Our algorithm \change{is based on a special case}{uses as a departure point}  the main result of%
\revone{4. Page 4, line 40: ``Our algorithm is based on...''  - what do you mean by this, do you only follow the same ideas or you use particular results? Or do you mean directly Theorem 1.5? Please specify this clearly.} 
\ans{The results of \cite{thilikos2024Wexcluding} serve as the departure point of our proofs. We clarify this in the revised version.}
{\cite{thilikos2024Wexcluding}} that reveals the structure of graphs that do not contain a handle wall nor a crosscap wall of order $k$ as subgraphs.
\rred{Using the results of \cite{thilikos2024Wexcluding}, we prove that} each such graph admits a tree decomposition in pieces that are ``$\bw$-almost planar'', in the sense that their branchwidth is not much bigger than the branchwidth of some of their planar subgraphs (\cref{from_surfex}).
Moreover, such a decomposition (for fixed values of $k$) can be constructed in  polynomial time.

\begin{theorem}\label{th_w1rgol}
There is a function $f_{\ref{th_w1rgol}}\colon\mathbb{N}\to\mathbb{N}$ and an algorithm that, given a graph $G$ and a non-negative integer $k,$ either outputs a  
handle wall or a crosscap  wall of $G$ of order $k$ or outputs a subgraph $G'$ of $G$ where $\bw(G)-\bw(G')\leq f_{\ref{th_w1rgol}}(k)$ and a tree decomposition $(T',\beta')$ of $G'$ such that, for every $t\in V(T'),$
\begin{itemize}
\item[(a)] the torso $G_{t}'$ is planar,
\item[(b)] the torso  $G_{t}'$ is a minor of $G',$ and 
\item[(c)] every adhesion set of $t$ has size at most three.
\end{itemize}
Moreover, $f_{\ref{th_w1rgol}}(k)=2^{\poly(k)}$ and the above algorithm runs in time $\Ocal(2^{2^{\poly(k)}}|V(G)|^3)$
\end{theorem}

For the formal definitions of tree decompositions, torsos, and adhesion sets, see \cref{dc_sefd}.
Notice that the tree decomposition $(T',\beta')$ of $G'$ in \cref{th_w1rgol} implies that $G'$ enjoys the structure of $K_{5}$-minor free graphs given by Wagner's theorem~\cite{Wagner37ubere}.
This permits the following more compact restatement of \cref{th_w1rgol}.

\begin{corollary}
For every $k,$ there exists some $r$ such that every graph that does not contain a handle wall or a crosscap wall of order $k$  as a subgraph, contains a $K_{5}$-minor free subgraph $G'$ where $\bw(G)-\bw(G')\leq r$.
\end{corollary}

The decomposition of \cref{th_w1rgol} permits us to call the polynomial algorithm %
\revone{5. Page 5, line 9: I suggest that when you cite papers (like [35], or [6],[33]) for results you use, it is polite to include authors' names with the citation.} 
\ans{Done, for all papers mentioned.}
of \rred{Seymour and Thomas}~\cite{SeymourT94callr} in each $G_{t}'$ and suitably combine the obtained values in order to design the constant-additive approximation for branchwidth claimed by \cref{arst_cpe}.  
The technical details of how to prove \cref{th_w1rgol}, and how to make algorithmic use of it, are given in \cref{secintroduction}. In \cref{sec_outline}
we give a brief outline on the combinatorial results supporting our algorithm. We conclude with some open problems and research directions in \cref{opne_all}.

\section{Decompositions and algorithms}\label{secintroduction}
Given two non-negative integers $a,b\in\mathbb{N}$ we denote the set $\{z\in\mathbb{N} \mid a\leq z\leq b\}$ by $[a,b].$
In case $a>b$ the set $[a,b]$ is empty.
For an integer $p\geq 1,$ we set $[p]\coloneqq [1,p]$ and $\mathbb{N}_{\geq p}\coloneqq\mathbb{N}\setminus [0,p-1].$
The graphs considered in this paper are undirected, finite,  without loops, and may have multiple edges, i.e., , the edge set is a multi-set.
In fact, multiple edges will be relevant in \cref{sec_outline} where we 
deal with spherical embeddings.
We use standard graph-theoretic notation, and we refer the reader to~\cite{diestel2016graph} for any undefined terminology.

\subsection{Parametric graphs}
\label{subsec_par_gr}

\rred{
A \emph{parametric graph} is a sequence of graphs 
$\mathscr{G}=\langle \mathscr{G}_{k}\rangle_{k\in \mathbb{N}}$. In this paper we consider 
parametric graphs that are \emph{minor-monotone}: for every 
$k\leq k'$, $\mathscr{G}_{k}$ is a minor of $\mathscr{G}_{k'}$.

Given a $(k\times 8k)$-cylindrical grid $\mathscr{G}_{k}$ 
we  say that a path in it is \emph{vertical} if it is a shortest 
path between two vertices of degree 3. Clearly, $\mathscr{G}_{k}$ has 
$8k$ vertical paths that are all pairwise disjoint and their union span 
all the vertices of $\mathscr{G}_{k}$. Let now $M_{k}$ be one of the two perfect matchings 
of $\mathscr{G}_{k}$ that consists of edges of vertical paths.
Notice that $\mathscr{G}_{k}$ contains two cycles consisting only of vertices 
of degree 3 and, among them, we pick one, and  $v_{1},\ldots,v_{8k}$
in order to denote its vertices in the order that they appear (the starting vertex $v_{1}$ is chosen arbitrarily).
We define the  \emph{elementary annulus wall}%
 as the    parametric graph $\mathscr{W}^{\manman{A}}=\langle \mathscr{W}_{k}^{\manman{A}}\rangle_{k\in \mathbb{N}}$, where $\mathscr{W}_{k}^{\manman{A}}$ is the graph obtained from $\mathscr{G}_{k}$ 
 after removing the edges in $M_{k}$. Clearly, this reduces by one 
 all degrees of the vertices of  $\mathscr{G}_{k}$, therefore 
 $\mathscr{W}^{\manman{A}}$ is a graph of maximum degree 3.
The  \emph{elementary handle wall}  $\mathscr{W}^{\manman{H}}=\langle \mathscr{W}_{k}^{\manman{H}}\rangle_{k\in \mathbb{N}}$ 
is the parametric graph where  $\mathscr{W}_{k}^{\manman{H}}$ is obtained 
from $\mathscr{W}_{k}^{\manman{A}}$ after adding all the edges 
in $\{v_{i}v_{4k+1-i}\mid i\in[k]\}$ and in $\{v_{k+i}v_{3k+1-i}\mid i\in[k]\}$.
Also the  \emph{elementary crosscap wall}  $\mathscr{W}^{\manman{C}}=\langle \mathscr{W}_{k}^{\manman{C}}\rangle_{k\in \mathbb{N}}$ 
is the parametric graph where  $\mathscr{W}_{k}^{\manman{C}}$ is obtained 
from $\mathscr{W}_{k}^{\manman{A}}$ after adding all the edges 
in $\{v_{i}v_{2k+i}\mid i\in[2k]\}$.
As we already explained in \cref{subsec_main_c}, 
an \emph{annulus wall}/\emph{handle wall}/\emph{crosscap wall} of \emph{order} $k$ is a subdivision of $\mathscr{W}_{k}^{\manman{H}}$/$\mathscr{W}_{k}^{\manman{A}}$/$\mathscr{W}_{k}^{\manman{C}}$, where 
the \emph{subdivision} of an edge is the result of the replacement of an edge 
by a path of arbitrary length and with the same endpoints.

Notice that the above definition does not define $\mathscr{W}_{k}^{\manman{H}}$/$\mathscr{W}_{k}^{\manman{A}}$/$\mathscr{W}_{k}^{\manman{C}}$ when $k=0,1,2$. For full consistency with the definition of a parametric graph we may consider that 
$\mathscr{W}_{k}^{\manman{H}}$/$\mathscr{W}_{k}^{\manman{A}}$/$\mathscr{W}_{k}^{\manman{C}}$ is the empty graph for these values.

  }

\vspace{-0mm}\subsection{Decompositions and tangles} \label{dc_sefd}

Many ``width parameters'' for graphs are defined by ``gluing'' together simpler graphs of some particular property in a tree-like fashion.
A possible way to formalize this notion of tree-likeness is given by tree decompositions.

\begin{definition}[Tree decompositions]\label{def_treewidth} 
  Let $G$ be a graph. 
  A \emph{tree decomposition} of $G$ is a tuple $(T,\beta)$ where $T$ is a tree and $\beta\colon V(T)\to 2^{V(G)}$ is a function such that 
  \begin{enumerate} 
    \item $\bigcup_{t\in V(T)} \beta (t)= V (G),$ 
    \item for every $e\in E(G)$ there exists $t\in V(T)$ with $e\subseteq \beta(t),$ and
    \item for every $v\in V(G)$ the set $\{t\in V(T) \mid v\in \beta(t)\}$ induces a subtree of $T.$
  \end{enumerate} 
  The \emph{torso} of $(T,\beta)$ on a node $t$ is the graph, \revthree{8. Why do we need that adhesions are small cliques?}
\ans{This is the standard definition of a torso of a node $t$ of a tree decomposition. Typically the adhesions of $t$ are made cliques in order to ``represent'' the fact that they are separating the bag corresponding to this node from the rest of the graph. Torsos are important in proofs when dealing with tree/branch decompositions or tangles.}
  denoted by  $G_t$, obtained by adding in $G[\beta(t)]$ all edges between the vertices of $\beta(t) \cap \beta(t')$ for every neighbor $t'$ of $t$ in $T.$
{For each $t\in V(T),$ we define the \emph{adhesion sets of $t$}
 as the sets in $\{ \beta(t)\cap\beta(d) \mid dt\in E(T) \}$ and the maximum size of an adhesion set of $t$ is called the \emph{adhesion of $t$}.}
 \rred{Clearly, in case $T$ has only one node, then its adhesion set is empty and we set the adhesion of $(T,\beta)$ to be $0$.}
\revone{6. Page 5, line 46: This is not quite correct when $T$ has one node.}%
\ans{We now explain that if $T$ has only one node, then its adhesion set is empty.}

The \emph{adhesion} of $(T,\beta)$ is the maximum adhesion of a node of $(T,\beta)$.
The \emph{treewidth} of $G,$ denoted by $\tw(G),$ is the minimum $k$ for which $G$ has a tree decomposition where all torsos have at most $k+1$ vertices.
\end{definition}

\begin{observation}\label{vocmoptreedec}
Let $G$ be a graph and let $(T,\beta)$ be a tree decomposition of $G$.
Then $\tw(G)\leq \max\{\tw(G_t)\mid t\in V(T)\}$.
\end{observation}

\paragraph{Branchwidth.}
A \emph{branch decomposition} of a graph $G$ is a pair $(T,\delta),$ where $T$ is a tree with vertices of degree one or three and $\delta$ is a bijection from $E(G)$ to the set of leaves of $T$.
The \emph{order function} $\omega\colon E(T)\rightarrow 2^{V(G)}$ of a branch decomposition maps every edge $e$ of $T$ to a subset of vertices $\omega(e)\subseteq V(G)$ as follows.
The set $\omega(e)$ consists of all vertices $v \in V(G)$ such that there exist edges $f_1, f_2 \in E(G)$ with $v\in f_1 \cap f_2,$ and such that the leaves $\delta(f_{1}),$ $\delta(f_{2})$ are in different components of $T-e$.

The \emph{width} of $(T,\delta)$ is equal to $\max_{e\in E(T)}|\omega(e)|$ and is zero, in case $|V(T)|=1$.
The \emph{branchwidth} of ${G},$ denoted by $\bw(G),$ is the minimum width over all branch decompositions of $G$.
If $E(G)=\emptyset,$ then $\bw(G)=0$.

The next two results reveal important algorithmic properties on branchwidth. 

\begin{proposition}[Ratcatcher algorithm \rred{of Seymour and Thomas} \cite{SeymourT94callr}]\label{pr_rat_k}
There exists an algorithm that, given a planar graph $G$, outputs $\bw(G)$ in time $\Ocal(|V(G)|^3)$
\end{proposition}

\begin{proposition}[\rred{Bodlaender and Thilikos}  \cite{BodlaenderT97const}]\label{pr_bodl_linear}
There exists a function $f_{\ref{pr_bodl_linear}}:\Nbbb\to\Nbbb$ and an algorithm that, given a graph $G$, outputs $\bw(G)$ in time $\Ocal(f_{\ref{pr_bodl_linear}}(\bw(G))\cdot |V(G)|)$, where $f_{\ref{pr_bodl_linear}}(k)=2^{\poly(k)}$
\end{proposition}

According to \rred{Robertson and Seymour in}~\cite{robertson1991graph}, for every graph $G$ with at least one edge it holds that $\max\{\bw(G),2\}\leq \tw(G)+1\leq \max\{ \lfloor \frac{3}{2}\bw(G)\rfloor,2\}$.
For our purposes we rewrite the inequality in a slightly more relaxed form as follows.

\begin{lemma}\label{brwibw}
For every graph $G,$ $\bw(G)\leq \tw(G)+1\leq \lfloor \frac{3}{2}\bw(G)\rfloor+2$.
\end{lemma}

\paragraph{Tangles.}
In~\cite{robertson1991graph}, Robertson and Seymour introduced tangles as the max-min counterpart of branchwidth.
Tangles have been important in the proofs of their seminal series of papers on graph minors.
Moreover, {\change{T}{t}angles}%
\revone{7. Page 6, line 26: No capital letter in ``Tangle'' .} 
\ans{Fixed.}
also play an important role in an abstract theory of connectivity, see for example~\cite{Grohe16bTangledup,DiestelK20profi,DiestelO19tangl,DiestelEW19struc,Diestel18abstr}.

\begin{definition}[Tangle]\label{def_tangle} 
A \emph{separation} in a graph $G$ is a pair $(A,B)$ of vertex sets such that $A\cup B=V(G)$ and there is no edge in $G$ with one endpoint in $A\setminus B$ and the other in $B\setminus A.$
The \emph{order} of $(A,B)$ is $|A\cap B|.$
Let $G$ be a graph and $k$ be a positive integer. 
We denote by $\mathcal{S}_k(G)$ the collection of all tuples $(A,B)$ where $A,B\subseteq V(G)$ and $(A,B)$ is a separation of order $<k$ in $G.$ 
An \emph{orientation} of $\mathcal{S}_k(G)$ is a set $\Ocal$ such that for all $(A,B)\in\mathcal{S}_k(G)$ exactly one of $(A,B)$ and $(B,A)$ belongs to $\Ocal.$ 
A \emph{tangle} of order $k$ in $G$ is an orientation $\mathcal{T}$ of $\mathcal{S}_k(G)$ such that 
{\begin{eqnarray}
\mbox{for all $(A_1,B_1),(A_2,B_2),(A_3,B_3)\in\mathcal{T}$ it holds that $G[A_1]\cup G[A_2]\cup G[A_3]\neq G.$}\label{tangleax}
\end{eqnarray}}
\end{definition}
In the above, we adopt the definition of tangles from \cite{KawarabayashiTW20Quicklyexcluding}.
\begin{proposition}[\cite{robertson1991graph}]\label{minmax_tangle}
Let $G$ be a graph where $\bw(G)\geq 2$.
Then the maximum order of a tangle of $G$ is equal to $\bw(G)$. 
\end{proposition}

The next observation is a direct consequence of \cref{minmax_tangle} and the definition of a tangle.

\begin{observation}\label{apex_bw}
Let $G$ be a graph   and let $X\subseteq V(G)$ such that $\bw(G-X)\geq 2$. 
Then $\bw(G-X)\geq \bw(G)-|X|$.
\end{observation}

In this paper we repetitively make use of \cref{minmax_tangle} in order to work with tangles instead of branch decompositions.
Notice that we only make use of the concept of tangles to prove the correctness of our algorithm.

\begin{lemma}\label{alltog}
Let $r\in\mathbb{N}_{\geq 1}$.
Let $G$ be a graph where {$\bw(G)\geq r+1$}%
\revone{8. Page 6, line 55: I suggest to shortly explain why $\bw(G)\geq r+1$ is needed.
}
\ans{This is because in \cref{minmax_tangle}  it is required that $\bw(G)≥2$
and in \cref{alltog} we have $r\in\mathbb{N}_{\geq 1}$. We don't think that it is necessary to explain this in the proof of \cref{alltog}.}
that has a tree 
decomposition $(T,\beta)$ of adhesion at most $r$.
Then the branchwidth of $G$ is at most the maximum branchwidth of the torsos of $(T,\beta)$.
\end{lemma}

\begin{proof}
Let $k\coloneqq\bw(G)\geq r+1\geq 2$.
By \cref{minmax_tangle} we may consider some tangle $\mathcal{T}$ of $G$ of order $k$.
Notice that the \change{edges of $(T,\beta)$}{edges of $T$ in the branch decomposition $(T,\beta)$}%
\revthree{6. page 7, line 2: what are edges of $(T,\beta)$?}
\ans{Fixed}
 naturally define \rred{separations $(A,B)$ from $\mathcal{S}_k(G)$ where 
 the sets $A\cap B$ are exactly the corresponding adhesion sets}.
Therefore, $\mathcal{T}$ induces an orientation \stra{$\vec{T}$} %
of the edges of $T$ \change{with the property that $\vec{T}$}{that has} has a unique sink $t$. 
We will use $\mathcal{T}$ to define a tangle $\mathcal{T}'$ of order $k$ for $G_{t}$. \revthree{7. page 7, line 5: T arrow is never used in the definition. In general the readability of this proof could be improved:}
\ans{We now avoid the use of $\vec{T}$. We also made some changes in the proof in order to improve its readability.}
Then the result follows from \cref{minmax_tangle}.

\begin{claim}\label{cliam_bigBag}
 \rred{$|\beta(t)|\geq k$ (where $k≥2$).}
\end{claim}

\rred{\begin{proof}[Proof of Claim.]
Suppose $|\beta(t)|<k$ where $k\geq 2$. Notice that $(\beta(t),V(G)) \in \mathcal{T}$ as otherwise this would already contradict \eqref{tangleax}.
Now, suppose $V(T)=\{t\}$.
In this case we have that $G_t=G$ and our claim follows trivially.
So we may assume that $t$ has at least one neighbor, say $d$, in $T$.
Let $(A_d,B_d)$ be the separation induced the edge $dt$ and assume w.l.o.g. that $\beta(t)\subseteq B_d$, then $(A_d,B_d) \in \mathcal{T}$ and $A_d\cap B_d \subseteq \beta(t)$.
Now let $N_T(t)=\{d_1,\dots,d_n\}$ and consider any non-empty set $J\subseteq [n]$.
Then $(\bigcup_{i\in J}A_{d_i},\bigcap_{i\in J}B_{d_i}) \in \mathcal{S}_k(G)$ where $\beta(t)\subseteq \bigcap_{i\in J}B_{d_i}$.

We claim that $(\bigcup_{i\in J}A_{d_i},\bigcap_{i\in J}B_{d_i})\in\mathcal{T}$.
This follows by induction on $|J|$ as follows.
For $|J|=1$ we have already provided an argument above.
So let $J\geq 2$ and let $j\in J$ be the largest index in $J$, we set $J' \coloneqq J\setminus\{ j\}$.
By induction, we know that $(\bigcup_{i\in J'}A_{d_i},\bigcap_{i\in J'}B_{d_i}) \in \mathcal{T}$ and we have that $(A_{d_j},B_{d_j})\in\mathcal{T}$ from the base case.
If now $(\bigcap_{i\in J}B_{d_i},\bigcup_{i\in J}A_{d_i})\in\mathcal{T}$ we have that $G[A_{d_j}] \cup G[\bigcup_{i\in J'}A_{d_i}] \cup G[\bigcap_{i\in J} B_{d_i}]=G$.
This contradicts \eqref{tangleax} in \cref{def_tangle}.
Hence, our claim follows.

Now, by the above  claim,  we have that $(X,Y) \coloneqq (\bigcup_{i\in[n]}A_{d_i},\bigcap_{i\in [n]}B_{d_i})\in\mathcal{T}$.
However, since for every vertex $v\in V(G)\setminus \beta(t)$ there is some $i\in[n]$ such that $v\in A_{d_i}$ we have that $(X,Y)=(X,\beta(t))\in\mathcal{T}$.
As $(\beta(t),V(G))\in \mathcal{T}$ this contradicts \eqref{tangleax} in \cref{def_tangle}.
{It follows that $|\beta(t)|\geq k$ as desired.}
\end{proof}}

For the construction of $\mathcal{T}',$ we consider every separation $(A',B')\in\mathcal{S}_{k}(G_t)$.
As all adhesion sets of $t$ in $T$ induce cliques of size $\leq r$ in $G_{t}$ and $k>r,$ there is a $(A,B)\in \mathcal{S}_{k}(G)$ such that $A'=A\cap V(G_{t})$, $B'=B\cap V(G_{t})$, and $A'\cap B'=A\cap B$.
Notice that this correspondence is not unique as there can be \change{soveral}{several}%

\revthree{9. soveral ~$\longrightarrow$~ several} 
\ans{Fixed}

such separations $(A,B)\in\mathcal{S}_k(G)$.
However, when restricted to $G_t$\rred{,} all of them become $(A',B')$.
Moreover, if $(A_1,B_1)$ and $(A_2,B_2)$ correspond to $(A',B')$ in the above way\rred{,} then $(A_1,B_1)\in \mathcal{T}$ if and only if $(A_2,B_2)\in\mathcal{T}$\rred{,} by our choice of $t$%
.%
\revthree{10. What if $A' = B' = \beta(t) = \{v\}$? Then the choice is not well-defined, right?}
\ans{We added Claim \ref{cliam_bigBag} together with a proof to take care of this. Now the case where $|\beta(t)|\leq 1$ is explicitly ruled out.}

We now define $\mathcal{T}'$ so that for every $(A',B')\in\mathcal{S}_{k}(G_t),$  $(A',B')\in \mathcal{T}'$ if and only if $(A,B)\in \mathcal{T}$ where $(A,B)$ is any corresponding separation of $G$ as described above.
Clearly $\mathcal{T}$ is an orientation of  $\mathcal{S}_{k}(G_t)$.
Moreover, again because all adhesion sets of $t$ induce cliques of size $\leq r$ in $G_{t},$ it follows that if for some $(A_1',B_1'),(A_2',B_2'),(A_3',B_3')\in\mathcal{T}'$ we have $G_{t}[A_1']\cup G_{t}[A_2']\cup G_{t}[A_3']= G_{t}$ then it also holds that $G[A_1]\cup G[A_2]\cup G[A_3]= G$.
Therefore, $\mathcal{T}'$ is a tangle of order $k$ in $G_{t},$ as required.
\end{proof}

Given a graph $G$ and a positive integer $k$ we define the graph  $G^{(k)}$ by replacing each vertex $v$ of $G$ with a $k$-clique $K_{v}$ and each edge $e=uv$ with all $k^2$ edges between the vertices of $K_{v}$ and the vertices of $K_{u}$.

\begin{lemma}\label{ooodblow}
Let $G$ be a graph where $\bw(G)\geq 2$ and $k\in\mathbb{N}_{\geq 1}$.
Then $\bw(G^{(k)})=k\cdot\bw(G)$.
\end{lemma}

\begin{proof}
By \cref{minmax_tangle}, we may prove that every tangle $\mathcal{T}$ in $G$ of order $t$ corresponds to a tangle $\mathcal{T}'$ in $G^{(k)}$ of order $kt$.

Suppose that $\mathcal{T}$ is a tangle in $G$ of order $t$.
Let now $(A',B')\in \mathcal{S}_{kt}(G^{(k)})$.
Notice that, for every $v\in V(G),$ the vertices of $K_{v}$ should be either all in $A'$ or all in $B'$. 
This means that we may consider a separation $(A,B)$ of $G$ where $A=\{v\in V(G)\mid V(K_{v})\subseteq A\}$ and $B=\{v\in V(G)\mid V(K_{v})\subseteq B\}$ that has order $k$.
By orienting each separation $(A',B')$ of $G^{(k)}$ in the same way as the corresponding $(A,B)$ in $\mathcal{T}$ we obtain a tangle of $G^{(k)}$ of order $kt$.
Suppose now that $\mathcal{T}'$ is a tangle in $G^{(k)}$ of order $kt$ and let $(A,B)$ be a separation of $G$ of order $t$.
Observe that $(A',B')=(\bigcup_{v\in B}V(K_{v}),\bigcup_{v\in A}V(K_{v}))\in\mathcal{S}_{kt}(G^{(k)})$.
As before, we may use the orientation of $\mathcal{S}_{kt}(G^{(k)})$ defined by $\mathcal{T}',$ in order to define an orientation $\mathcal{T}$ of $\mathcal{S}_{t}(G)$ that is a tangle of $G$ of order $t$.
\end{proof}

\cref{ooodblow} permits us to prove that one may not expect, in general, any constant-additive approximation for branchwidth that runs in polynomial time.
The following proof of \cref{lower} uses a reduction employed by~\cite{BodlaenderGHK95appro} for deriving the analogous result for treewidth.

\medskip\noindent \textbf{\cref{lower}} (\textsl{restated}). {\midnightblack{{Branchwidth does not admit a polynomial time additive approximation algorithm, unless $\mathsf{P}=\mathsf{NP}$.}}

\begin{proof}
\revone{9. Page 7, line 49: I suggest to restate the theorem before its proof. The same suggestion for subsequent proofs.}
\ans{Done. We now add restatements before all postponed proofs.}
Suppose that there is some integer $k\geq 0$ and an algorithm that, for every graph $G,$ outputs a number $b$ such that $\bw(G)\leq b\leq \bw(G)+k,$ for some constant $k\geq 0$.
We use this algorithm as a subroutine to compute $\bw(G)$.
If $\bw(G)<2,$ then we may compute $\bw(G)$ using \cref{pr_bodl_linear}.
If $\bw(G)\geq 2,$ we apply the additive approximation algorithm on $G^{(k+1)}$ and, by calling upon \cref{ooodblow}, obtain a value $b$ where $\bw(G^{(k+1)})\leq b\leq \bw(G^{(k+1)})+k\Rightarrow\bw(G)\leq b/(k+1)\leq \bw(G)+1-\frac{1}{k+1}<\bw(G)+1,$ which implies that $\bw(G)=\lfloor b/(k+1)\rfloor$.
This would make possible to decide {in polynomial time}  whether, given a graph $G$ and an integer $r$,  $\bw(G)\leq r$, which, as proven \rred{by Seymour and Thomas \cite{SeymourT94callr}}, is an \textsf{NP}-complete problem.
\end{proof}

Notice that the blow-up operation used to define $G^{(k)}$ introduces large cliques to the graph.
This means that the operation does not respect the exclusion of a minor which means that \cref{ooodblow} and \cref{lower} do not preclude additive approximation algorithms for the branchwidth of graphs excluding a minor.

\vspace{-0mm}\subsection{Approximation algorithms}
The combinatorial engine of our approximation algorithms is the following result.

\begin{theorem}\label{maintscomb}
There exists a function $f_{\ref{maintscomb}}\colon\mathbb{N}\to\mathbb{N}$ and an  algorithm that, given a graph $G$ and an integer $k,$ either outputs a subgraph of $G$ that is a handle wall or a crosscap wall of order $k$ or outputs a tree decomposition $(T,\beta),$ where every torso $G_{t}$, $t\in V(T)$, contains a set $X_{t}$ such that
\smallskip
\begin{enumerate}
\item  $G_{t}-X_{t}$ is planar,
\item  $G_{t}-X_{t}$ is a minor of $G,$  
\item  every adhesion of $t$ contains at most $f_{\ref{maintscomb}}(k)$ vertices and, among them, at most three do not belong to $X_{t},$ and 
\item $\bw(G_t)-\bw(G_{t}-X_{t})\leq f_{\ref{maintscomb}}(k)$.
\end{enumerate}
Moreover, $f_{\ref{maintscomb}}(k)=2^{\poly(k)}$ and the above algorithm runs in time $\Ocal(2^{2^{\poly(k)}}|V(G)|^3)$. 
\end{theorem}

We postpone the proof of \cref{maintscomb} to \cref{sec_outline}.
\cref{maintscomb}, is used for the proof of \cref{th_w1rgol}, as follows.

\medskip\noindent\textbf{\cref{th_w1rgol}} (\textsl{restated}). {\midnightblack{{There is a function $f_{\ref{th_w1rgol}}\colon\mathbb{N}\to\mathbb{N}$ and an algorithm that, given a graph $G$ and a non-negative integer $k,$ either outputs a  
handle wall or a crosscap  wall of $G$ of order $k$ or outputs a subgraph $G'$ of $G$ where $\bw(G)-\bw(G')\leq f_{\ref{th_w1rgol}}(k)$ and a tree decomposition $(T',\beta')$ of $G'$ such that, for every $t\in V(T'),$
\begin{itemize}
\item[(a)] the torso $G_{t}'$ is planar,
\item[(b)] the torso  $G_{t}'$ is a minor of $G',$ and 
\item[(c)] every adhesion set of $t$ has size at most three.
\end{itemize}
Moreover, $f_{\ref{th_w1rgol}}(k)=2^{\poly(k)}$ and the above algorithm runs in time $\Ocal(2^{2^{\poly(k)}}|V(G)|^3)$.}}

\begin{proof}
Consider a tree decomposition $(T,\beta)$ of $G$ where every torso $G_{t}$ contains a set $X_{t}$ such that conditions i), ii), iii), and iv) are satisfied.
By \cref{maintscomb}, such a decomposition can be computed in time $\Ocal(2^{2^{\poly(k)}}|V(G)|^3)$.
For each $t\in V(T),$ we set $G_{t}'\coloneqq G_{t}-X_{t}$.
Let $X\coloneqq \bigcup_{t\in V(T)}X_{t}$ and $G'\coloneqq G-X$.
We consider the tree decomposition $(T,\beta')$ of $G'$ where, for each $t\in V(T),$ $\beta'(t)=\beta(t)\setminus X_{i}$ (clearly $\beta'(t)=V(G'_{t})$).
The algorithm outputs $G'$ and $(T,\beta')$.
For every $t\in V(T),$ conditions (a), (b), and (c) hold for $(T,\beta')$ because conditions i), ii), and iii) of \cref{maintscomb} respectively hold for $(T,\beta')$.
It remains to prove that
\begin{eqnarray}
  \bw(G)-\bw(G')\leq f_{\ref{maintscomb}}(k). \label{lpopi}  
\end{eqnarray}

If $\bw(G)\leq f_{\ref{maintscomb}}(k),$ we directly have that $\bw(G)-\bw(G')\leq \bw(G)\leq f_{\ref{maintscomb}}(k)$ and~\eqref{lpopi} follows.

If $\bw(G)> f_{\ref{maintscomb}}(k),$ then, using \cref{maintscomb}.iii), we may apply \cref{alltog} for $G$ and $r=f_{\ref{maintscomb}}(k)$ and obtain that $\bw(G)\leq \max\{\bw(G_{t}) \mid t\in V(T)\}\leq^{\text{\cref{maintscomb}.iv)}}  \max\{\bw(G_{t}') \mid t\in V(T)\}
+f_{\ref{maintscomb}}(k)\leq^{\text{\change{(b)}{\cref{maintscomb}.iii)}}} \bw(G')+f_{\ref{maintscomb}}(k)$. This yields \eqref{lpopi} and the result follows by setting $f_{\ref{th_w1rgol}}(k)=f_{\ref{maintscomb}}(k)=2^{\poly(k)}$.%
\revthree{11. page 8, line 45: Theorem 2.11 ii) instead of (b)?}
\ans{Indeed. Fixed.}

\end{proof}

Using now \cref{th_w1rgol}, we can prove \cref{arst_cpe}.

\medskip\noindent \textbf{\cref{arst_cpe}} (\textsl{restated}).~{\midnightblack{{Branchwidth admits a polynomial-time additive approximation algorithm for every graph class that excludes as a minor a graph    $H_{1}\in \mathcal{E}_{1}$ and a graph $H_{2}\in \mathcal{E}_{2}$. 
In particular, there is a function $f_{\ref{arst_cpe}}\colon\mathbb{N}\to\mathbb{N}$ and an algorithm that, given as input two graphs  $H_{1}\in \mathcal{E}_{1}$ and $H_{2}\in \mathcal{E}_{2},$ and an $\{H_{1},H_{2}\}$-minor free graph $G,$ outputs a value $b$ where $b\leq \bw(G)\leq  b+f_{\ref{arst_cpe}}(k),$ where $k=|V(H_{1})|+|V(H_{2})|$. Moreover, $f_{\ref{arst_cpe}}(k)=2^{\poly(k)}$ and the algorithm runs  in time $\Ocal(2^{2^{\poly(k)}}|V(G)|^{3})$.}}

\begin{proof}
By \cref{kessent_exl}, $G$ does not contain a handle wall or a crosscap wall of order 
\change{$k'=c|V({H}|)^2$}{$k'=c{k}^2$} as a subgraph.%
\revthree{12. page 8, line 52: there is no H in the statement of Theorem 1.1.}
\ans{We now set $k'=ck^2$, after applying \cref{kessent_exl}.}
Thus, using \cref{th_w1rgol}, we can find, in time $\Ocal(2^{2^{\poly(|V(H)|)}}|V(G)|^3)$, a subgraph $G'$ of $G$ where 
\begin{eqnarray}
\bw(G) & \leq & \bw(G')+f_{\ref{th_w1rgol}}(k\rred{'}) \label{onmoere}
\end{eqnarray}
and a tree decomposition $(T',\beta')$ of $G'$ such that, for every $t\in V(T'),$ conditions (a), (b), and (c) of \cref{th_w1rgol} hold. 
From condition (a) it follows that each torso $G_{t}', t\in V(T)$, is a planar graph.
Therefore we can use the algorithm of \cref{pr_rat_k} in order to compute the branchwidth of $G_{t}'$ in time $\Ocal(|G_{t}'|^3)$.
Our algorithm returns the value $b=\max \{\bw(G_{t}')\mid t\in V(T')\},$ that can be computed in time $\Ocal(|V(G')|)$.
We next claim that 
\begin{eqnarray}
\bw(G') & \leq& b+3 \label{thelastosp}
\end{eqnarray}
In order to prove~\eqref{thelastosp}, we distinguish two cases.

\noindent\emph{Case 1}:
$b>3$: In this case we may use property $(c)$ in order to apply \cref{alltog} for $G'$ and $r=3$ and obtain that $\bw(G')≤b<b+3$. 

\noindent\emph{Case 2}:
$b\leq 3$: Here, because of \cref{brwibw} and \cref{vocmoptreedec}, we have that $\bw(G')\leq \tw(G')+1\leq  \max \{\tw(G_{t}')\mid t\in V(T)\}+1\leq \max \{\lfloor \frac{3}{2}\bw(G_{t}')\rfloor+2\mid t\in V(T)\}= \lfloor \frac{3}{2}b\rfloor+2 \leq b+3$.  

We conclude that $b\leq^{\text{\cref{th_w1rgol}.b)}} \bw(G')\leq \bw(G)\leq^{\text{\eqref{onmoere}}} \bw(G')+f_{\ref{th_w1rgol}}(k\rred{'})  \leq^{\text{\eqref{thelastosp}}} b+3+f_{\ref{th_w1rgol}}(k\rred{'})$ and the theorem holds if we set $f_{\ref{arst_cpe}}=3+f_{\ref{th_w1rgol}}(k\rred{'})=2^{\poly(k\rred{'})}$.
\end{proof}

We now use \cref{th_w1rgol} for the EPTAS  of \cref{mainth}.

\medskip\noindent \textbf{\cref{mainth}} (\textsl{restated}). \textit{\midnightblack{Branchwidth admits an EPTAS for every graph class that excludes as a minor a graph $H_{1}\in \mathcal{E}_{1}$ and a graph $H_{2}\in \mathcal{E}_{2}$.
In particular, there is a function $f_{\ref{mainth}}\colon\mathbb{N}\to\mathbb{N}$ and an algorithm that, given as input two graphs  $H_{1}\in \mathcal{E}_{1}$ and $H_{2}\in \mathcal{E}_{2},$ and an $\{H_{1},H_{2}\}$-minor free graph $G,$ outputs a value $b$ where $b\leq \bw(G) \leq (1+\epsilon)b$. Moreover, the algorithm runs in time  $\Ocal(2^{2^{\poly(k)}}|V(G)|^{3}+f_{\ref{mainth}}(\epsilon,k)|V(G)|)$, where  $k=|V(H_{1})|+|V(H_{2})|$ and $f_{\ref{mainth}}(\epsilon,k)=2^{\poly(\frac{1}{\epsilon}\cdot 2^{\poly(k)})}$.
}}

\begin{proof}
Let $\epsilon> 0$.
We give an $(1+\epsilon)$-approximation algorithm as follows.
First, we run the algorithm of \cref{arst_cpe} and obtain some value $b$ such that $b\leq \bw(G)\leq b+f_{\ref{arst_cpe}}(k)$ where  $k=|V(H_{1})|+|V(H_{2})|$.
This algorithm runs in $\Ocal(2^{2^{\poly(k)}}|V(G)|^{3})$ steps.

If $f_{\ref{arst_cpe}}(k)/b<\epsilon,$ then $\bw(G)\leq b+f_{\ref{arst_cpe}}(k)< (1+\epsilon)b$.
The algorithm outputs $b$.
 
If  $f_{\ref{arst_cpe}}(k)/b\geq  \epsilon,$ this means that $\bw(G)\leq b+f_{\ref{arst_cpe}}(k)\leq f_{\ref{arst_cpe}}(k)/\epsilon+f_{\ref{arst_cpe}}(k)=(1/ \epsilon+1)\cdot f_{\ref{arst_cpe}}(k),$ and we use \cref{pr_bodl_linear} in order to output the exact value of $\bw(G)$ in time $\Ocal(f_{\ref{pr_bodl_linear}}((1/ \epsilon+1)\cdot f_{\ref{arst_cpe}}(k))\cdot |V(G)|)$.

In both cases, we output a value $b^*$ where $b^*\leq \bw(G) \leq (1+\epsilon)b^*,$ i.e., $b^*$ is an $(1+\epsilon)$-factor approximation of $\bw(G)$.
The result follows if we set $f_{\ref{mainth}}(\epsilon,k)=f_{\ref{pr_bodl_linear}}((1/ \epsilon+1)\cdot f_{\ref{arst_cpe}}(k))=2^{\poly(\frac{1}{\epsilon}\cdot 2^{\poly(k)})}$.
\end{proof}

\section{Proof for \cref{maintscomb}}\label{sec_outline}

In this section we give the proof for \cref{maintscomb}. For this we introduce some concepts from graph minors.

\subsection{Almost spherical embeddings}
Almost embeddings in surfaces form a key building block for the Graph Minors Structure Theorem of Robertson and Seymour {\cite{robertson2003graph}}.
For the purposes of our proofs, we only need the version of this concept when the  surface is the sphere $\Sbbb^2=\{(x,y,z)\mid x^2+y^2+z^2=1\}$.

A {\emph{spherical embedding}}%
\revone{10. Page 9, line 48+: I find this definition of an embedding to be quite unorthodox, and problematic in places (why not to use the same one as classical [Mohar-Thomassen]?). In what topological sense are your components ``connected''? Is this definition really giving edges as we know them even in boundary cases?}
\ans{We now explain that we use the notion of ``drawings in a surface'' Kawarabayashi, Wollan, and Thomas \cite{KawarabayashiTW20Quicklyexcluding}, restricted on spheres. This we believe, is more suitable for stating our results and proofs because we adapt their notion of $\Sigma$-decompositions which is based on their notion of drawings. We added some text to explain the origin oh these definitions.}
is  a pair  $(\Gamma, U)$ where  $U\subseteq \Gamma \subseteq \Sbbb_{0}$, $\Gamma$ is a closed set, $U$ is a finite set of points of $\Sbbb^2$, and each connected \rred{arc-wise} component of $\Gamma \setminus U$ is homeomorphic to the open set $(0,1)$.

We refer to the points in $U$ as the \emph{points} of  $(\Gamma, U)$
and to the connected components of $\Gamma \setminus U$ as the  \emph{arcs} of  $(\Gamma, U)$.
Given an arc $a$ of $(\Gamma, U)$, we define its endpoints as the points of  $\overline{a}\setminus a$ (we use the notation $\overline{a}$ for the 
closure of a set $a$). 
We say that  $(\Gamma, U)$ is a \emph{spherical embedding} of a graph $G$ if the graph $(U,\{xy\mid x \text{ and } y \text{ are the endpoints of  some arc in } (\Gamma, U) \}$ is isomorphic to $G$.
For our purposes, we consider only spherical embeddings where each arc has exactly two endpoints. 
That way we may consider spherical embeddings of graphs with multiple edges but not with loops.
Clearly, there is an one-to-one correspondence between the vertices/edges of $G$ and the points/arcs of $(\Gamma, U)$.
The set of faces of $(\Gamma, U)$ is the set of all the connected components of $\Sbbb^2\setminus \Gamma $ and is denoted by $F(\Gamma, U)$.

\rred{The concept of spherical embedding that we introduce here is based on the notion of ``drawings in a surface'' as used in their proof of the Graph Minor Structure Theorem by Kawarabayashi, Wollan, and Thomas \cite{KawarabayashiTW20Quicklyexcluding}.
In the context of graph minors one often works with the concept of an ``almost embedding'' to capture unavoidable impurities caused by low-order separations or ``vortices''.
Below we provide a simplified adaptation of the notion of ``$\Sigma$-decompositions'' as introduced by Kawarabayashi, Wollan, and Thomas \cite{KawarabayashiTW20Quicklyexcluding} based on the notion of (spherical) embeddings as introduced above.}

Given a spherical embedding $(\Gamma, U)$ of a graph $G$  and given a subset $N$ of the sphere that is homeomorphic to the \change{cycle}{circle}%
\revthree{16. page 11, line 56: cycle ~$\longrightarrow$~ circle}
\ans{Done.}
  $\{(x,y)\mid x^{2}+y^{2}=1\}$, we say that $N$ is a \emph{noose} of $(\Gamma, U)$ if $N\cap \Gamma \subseteq U$, i.e., ,  $N$ intersects the spherical embedding of $G$ only \change{to}{in}%
\revone{11. Page 10, line 6: ``intersects the spherical embedding of G only IN vertices'' } 
\ans{Fixed.}
vertices. We denote this set of vertices by $V_{G}(N)$.

\begin{definition}[Almost spherical embedding]
\label{defonfoe}
An \emph{almost spherical embedding} of a graph $G$ is a quadruple 
$(G',\Gamma ,U,\Bcal,\Pcal)$
where 
\begin{enumerate}
  \item   $G'$ is a subgraph of $G$ and  $\Bcal=\{B_1,\dots,B_{b}\}$ are  subgraphs of $G,$ such that $G=G'\cup\cupall \Bcal.$
  \item  $(\Gamma, U)$ is a spherical embedding of $G'$.
  \item 
  for each $i\in[b],$ there exists a noose $N_{i}$ of $(\Gamma, U)$  where $V(G')\cap V(B_i)=V_{G'}(N_i)=\{v^i_1,\dots,v^i_{n_i}\}$,  ordered according to the appearance of the vertices in $V_{G'}(N_i)$ while traversing along $N_{i}$  and such that one, say $\Delta_{i}$, of the two open disks bounded by $N_i$ does not contain any arc or point of $(\Gamma, U)$. Moreover, it holds that for $1\leq i<j\leq b$, $\Delta_{i}\cap \Delta_{j}\subseteq U$.
  \item   $\Pcal=\{(P_i,\beta_i)\mid i\in[b]\}$, where each 
  $(P_i,\beta_i)$  is a  path decomposition\footnote{A \emph{path decomposition} of a graph $G$ is a tree decomposition $(P,\beta)$ for $G$ where $P$ is a path.} of  $B_i$ such that
  \begin{enumerate}
    \item  $P_i$ is a path on the vertices $p^i_1, \dots, p^i_{n_i},$ appearing in this order, and
    \item  for each $j\in [n_i]$ we have $v^i_j\in \beta_i(p^i_j).$
  \end{enumerate}
\end{enumerate}
We call the graphs $B_i,$ $i\in [b],$ the \emph{vortices} of $(G',\Gamma ,U,\Bcal,\Pcal)$ and we call  the noose $N_i$ (resp. open disk $\Delta_{i}$) the \emph{boundary noose} (resp. \emph{disk}) of the vortex $B_i, i\in[b]$.
We call the vertices in $B_{i}\setminus V_{G'}(N_i)$ (resp. $V_{G'}(N_i)$) \emph{internal} (resp. \emph{boundary}) vertices of the vortex $B_{i}$.  
The \emph{width} of an almost spherical embedding $(G',\Gamma ,U,\Bcal,\Pcal)$ of $G$ is the maximum width over the path decompositions in $\Pcal$ and its \emph{breadth} is $b$, i.e., is the number of vortices.

For every vertex $v\in V(G'),$ we define $I_{v}$ as the set containing  every index $i\in[b]$
such that $v=v_j^i,$ for some $j\in[n_i]$, i.e., $I_{v}$ contains the indices 
of all the vortices that have $v$ as a boundary vertex.
If $I_{v}$ is empty, then we call $v$ \emph{poor}, otherwise we call it \emph{rich}.
\end{definition}
Notice also that $|I_v|\leq b$. \change{Notice that}{ Also,}  if $v$ is poor, then $V_{v}=\{v\}$.
For every $v\in V(G')$ we {define}%
\revthree{13. page 10, line 33: ``Notice that...''  should be after the definition of $V_v$.} 
\ans{Done.}
\begin{eqnarray}
V_{v} & = & \{v\}\cup\cupall\{\beta_i(v_{j}^{i})\mid i\in I_{v} ~\mbox{and $v=v_j^i$}\}.\label{dfkol_dm}
\end{eqnarray}

\rred{An \emph{$r$-dominating set} in a graph $G$ is a set $S\subseteq V(G)$ 
where every vertices of $G$ are within distance at most $r$ from some vertex of $S$.
We need the following result that was proved  in \cite[Theorem 2]{Thilikos16cover}.

\begin{proposition}
\label{prop_dom_planar}
If $G$ is a planar graph containing an $r$-dominating set of size $≤b$, 
then $\bw(G)≤(2r+1)\sqrt{4.5\cdot  b}$.
\end{proposition}
}

For our proofs we need to make the following observation that follows easily from known results.
\begin{lemma}
\label{ope_tyop}
There exists a function $f_{\ref{ope_tyop}}:\mathbb{N}^2\to\mathbb{N}$
such that if  $(G',\Gamma ,U,\Bcal,\Pcal)$ is an almost spherical embedding of a graph $G$ of width at most $w$ and breadth $b$ and where $\bw(G')\leq 1$, then  $\tw(G)\leq f_{\ref{ope_tyop}}(w,b)$.
\end{lemma}

\begin{proof}
It is easy to see that if $\bw(G')\leq 1$, then every connected component
of $G'$ is a star, i.e., , a tree of diameter at most two (see \cite[(4.4).(ii)]{robertson1991graph}).
We construct a new graph $\tilde{G}'$ by adding to $G'$, for every $i\in[b]$\rred{,} all the edges of $\{v^i_1v^i_2,\dots,v^i_{n_i-1}v^i_{n_i},v^i_{n_i}v^i_{1}\}$, that are not already present in $G'$.
We also construct $\tilde{G}$ by adding the same edges in $G$.

Our first step is to bound the treewidth of $\tilde{G}'$.
For this we add, for every $i\in[b]$, a new vertex $v_{i}^{\mathsf{new}}$ an make it adjacent to the vertices of $\{v^i_1,\dots,v^i_{n_i}\}$ and denote by $\tilde{G}^{\prime+}$ the resulting graph.
Observe that $\tilde{G}^{\prime+}$ is planar and that every vertex of some of its connected components that is not a star is within distance at most three
\revthree{14. page 10, line 52: I don't understand why we need three and two is not enough.} 
\ans{We went over this part and could not agree if two is enough of not. What we can agree on is that three defiantly works and since this is a tiny constant, we decided to keep it at three.}
from some vertex of $\{v^i_1,\dots,v^i_{n_i}\}$.
\revone{12. Page 10, line 54: I suggest to state what is in ``\cite[Theorem 2]{Thilikos16cover}' .} 
\ans{Done.}
{This}, according to \change{\mbox{\cite[Theorem 2]{Thilikos16cover}}}{\cref{prop_dom_planar}}, implies that 
\change{$\bw(\tilde{G}^{\prime})\leq \bw(\tilde{G}^{\prime+})\leq  {5}\sqrt{4.5\cdot  b}$}{
$\bw(\tilde{G}^{\prime})\leq \bw(\tilde{G}^{\prime+})\leq  {7}{5}\sqrt{4.5\cdot  b}$}
.

By \cref{brwibw} there exist a a tree decomposition $(T,\beta)$  of $\tilde{G}^{\prime}$ of width at most $k:=\frac{21}{2}\sqrt{4.5\cdot  b}+2$.

We  build a  tree decomposition $(T,\hat{\beta})$ of $\tilde{G}$ by recycling an argument that appeared in \cite{DemaineFHT05Subexponential} as Lemma 5.8. 
We set, for every $t\in V(T)$, $\hat{\beta}(t)\coloneqq\bigcup_{v\in \beta(t)}V_{v}$. By \eqref{dfkol_dm}, for every $v\in V(G')$, it holds that $|V_{v}|\leq w\cdot b+1$.
Therefore, $(T,\beta')$ has width at most $(k+1)(w\cdot b+1)-1$.
We now set $f_{\ref{ope_tyop}}(w,b)=(k+1)(w\cdot b+1)-1$ and conclude that $\tw(G)\leq f_{\ref{ope_tyop}}(w,b)$ as required.
\end{proof}

We now present the following result that can be {derived}
\revtwo{1. on page 11, line 5, \cite[Lemma 5.5]{thilikos2024Wexcluding} is referenced, but \cite{thilikos2024Wexcluding} does not contain a lemma 5.5. It seems that Theorem 7.1 might have been meant.}
\ans{This was an artifact from a problem we had with arxiv when we split the paper under \cite{thilikos2024Wexcluding} in two. This problem is now be resolved as our appeal was successful and we now cite the correct version. The paper is available at \href{https://arxiv.org/abs/2601.19230}{https://arxiv.org/abs/2601.19230}.}
from \cite[Lemma 5.5]{thilikos2024Wexcluding} for the case where the surface is a sphere. 
\revone{13. Page 11, line 5: Again, I suggest to state what is in ``\cite[Lemma 5.5]{thilikos2024Wexcluding}''  and restate your corollary.}
\ans{We would like to avoid restating this theorem as its full statement requires the full terminology of the Graph Minor Structure Theorem which would blow up the length of this paper by possibly several pages with comparatively small benefit.}
\begin{proposition}
\label{from_surfex}
There exists a function $f_{\ref{from_surfex}}:\mathbb{N}\to\mathbb{N}$ and an algorithm that, given a graph $G$ and  integer $k$, either finds a subgraph of $G$ that is either a handle wall or a crosscap wall of order $k$ or outputs a tree decomposition $(T,\beta)$ of $G$ of adhesion less than $f_{\ref{from_surfex}}(k)$ such that, for every $t\in V(T),$ assuming that $G_t$ is the torso of $G$ at $t\in V(T)$, it holds that either 
 \begin{itemize}
     \item  $|V(G_{t})|\leq f_{\ref{from_surfex}}(k)$ or 
     \item   there exist a set $A_t\subseteq\beta(t)$, called \emph{apex set},  where $|A_t|\leq f_{\ref{from_surfex}}(k)$ such that the graph ${G}_{t}'\coloneqq G_{t}-A_t$ has an almost spherical embedding  $(G_t',\Gamma _t,U_{t},\Bcal_t,\Pcal_t)$ of width at most $f_{\ref{from_surfex}}(k)$ and breadth at most $f_{\ref{from_surfex}}(k),$ such that 
  \begin{itemize}
  \item[B1.] for every neighbor $\hat{t}$ of $t$ in $T$, $G_t'\cap G_{\hat{t}}'$ is a complete graph  on at most three vertices and
  \item[B2.] $G_t'$ is a minor of $G.$
  \end{itemize}
 \end{itemize} 
 Moreover, $f_{\ref{from_surfex}}(k)=2^{\poly(k)}$ and the above algorithm runs in time $\Ocal(2^{2^{\poly(k)}}|V(G)|^{3})$. 
 \end{proposition}

We stress that \cref{from_surfex} is a special case of the main structural result in \cite{thilikos2024Wexcluding} that provides tree decompositions of spherical embeddings instead of similar decompositions for general surfaces. 
Also \change{a}{an}%
\revthree{15. page 11, line 26: ``Also a improved...''  ~$\longrightarrow$~ ``Also an improved...'' }
\ans{Fixed.}
 improved version of \cref{from_surfex} where $f_{\ref{from_surfex}}(k)=\poly(k)$
can be derived by the recent results of \cite{GorskyW25Polynomial} with an algorithm that runs in $\Ocal(2^{\poly(k)}{|V(G)|^6})$.
It also follows from the results in \cite{Thilikos2025TheGraph} that 
\cref{from_surfex} is tight in the sense that no such tree decomposition is possible for graphs containing a large handle wall or a crosscap wall. 

\begin{remark}
\label{remksiop}
In \cref{from_surfex} we may also assume that, in the second outcome, we also have 
\begin{itemize}
\item[] 
\begin{itemize}
\item[B3.] $\bw(G_{t}')\geq 2$.
\end{itemize}
\end{itemize}
\end{remark}
Indeed, for the $t$'s where $\bw(G_t')\leq 1$, we may use \cref{ope_tyop}, in order to construct a tree decomposition of $G_{t}$ of adhesion $\leq f_{\ref{ope_tyop}}(q)$ where each torso has at most $f_{\ref{ope_tyop}}(q)+1$ vertices, assuming that $q=f_{\ref{from_surfex}}(k)$.
Then we use these decompositions in order to transform the tree decomposition of  \cref{from_surfex}, where we may update $f_{\ref{from_surfex}}(k)$ to be $f_{\ref{ope_tyop}}(q)$.
This is possible because each adhesion set $X$ of $\beta(t)$ is a clique in $G_t$ and thus, the tree decomposition constructed above must have some bag which fully contains $X$.

\vspace{-0mm}\subsection{From tangles to slopes}
\label{this_che_slopes}

One may already see that the exclusion of handle walls and crosscap walls as subgraphs implies a decomposition into graphs that are ``almost'' planar (see \cref{from_surfex}).
Towards proving \cref{maintscomb}, we need to show that the the removal of the additional infrastructure of the apices and the internal vertices of the vortices, does not drop the branchwidth more than an additive constant.
For this we need the concept of slopes that is the min-max analogue of branchwidth (especially) for planar graphs. 

Let $(\Gamma, U)$ be a spherical embedding of a graph $G$.
Given a cycle $C$ of $G$, observe that the union $\Gamma_{C}$ of all those points and arcs of $(\Gamma,U)$ which correspond to the vertices and edges of $C$ is homeomorphic to the cycle $\{(x,y)\mid x^{2}+y^{2}=1\}$.
Therefore $\Sbbb\setminus \Gamma _{C}$ has two components that are both open disks of $\Sbbb$.
We say that the closures of these two open disks are the \emph{closed disks of $\Sbbb^2$ bounded} by $C$ in  $(\Gamma,U)$.
For simplicity, given a spherical embedding $(\Gamma,U)$ of $G$, we do not make any distinction between a cycle $C$ of $G$ and the closed curve $\Gamma_{C}$ in $(\Gamma,U)$ and we refer to both of them as the cycle $C$.

\begin{definition}[Slopes]
\label{def_slopes}
Let $(\Gamma, U)$ be a spherical embedding of a graph $G$   and let $k$ be a positive \rred{even} integer.
A \emph{slope} in $(\Gamma, U)$ of \emph{order} {$k/2$}%
\revone{14. Page 12, line 6: Do you expect $k/2$ to be an integer or not?} 
\ans{Yes. We now clarify this.}
is a function $\mathsf{ins}$ that maps each cycle $C$ of $G$ of length $<k$ to one of the two closed disks of $\Sbbb^2$ bounded by $C,$ such that
\begin{itemize}
\item[S1.] if $C$ and $C'$ are cycles of length $<k,$ and $C\subseteq \mathsf{ins}(C'),$ then $\mathsf{ins}(C)\subseteq \mathsf{ins}(C'),$
\item[S2.] if $P_1,P_2,P_3$ are  paths of $G$ joining two distinct vertices $x$ and $y$ and the three cycles $P_{1}\cup P_{2},$ $P_{2}\cup P_{3},$ and $P_{3}\cup P_{1}$ have length $<k,$ then 
$$\mathsf{ins}(P_{1}\cup P_{2})\cup\mathsf{ins}(P_{2}\cup P_{3})\cup\mathsf{ins}(P_{3}\cup P_{1})\neq \Sigma .$$
\end{itemize}
\vspace{-1mm}

We also say that a slope is \emph{uniform} if 
\begin{itemize}
\item[S3.] for every face $f\in F(\Gamma, U),$ there is a cycle of $G$ with length $<k$ such that $f\subseteq\mathsf{ins}(C)$.
\end{itemize}
\end{definition}

A graph is $2$-edge connected if it is connected and bridgeless.
Let $(\Gamma, U)$ be a spherical embedding of a  $2$-edge connected graph $G$. 
Notice that every face $f$ of $F(\Gamma, U)$ defines \change{a}{some} {cyclic}%
\revthree{17. page 12, line 21: if $\pi(f)$ is cyclic, then there are two options, we should choose one.} 
\ans{The choice of the ordering can be taken arbitrarily.}
ordering $\pi (f)$ of its incident vertices, possibly with repetitions, following the order that the corresponding points appear on the boundary of $f$.
A vertex $v$ incident to $f$ {appears}%
\revone{15. Page 12, line 23: ``A vertex $v$ incident to $f$ appears in $\pi(f)$ as many times as the number of connected components of $G - v$'' - this is not true in general, e.g., when $v$ repeats in several faces around $v$.} 
\ans{We believe what we write here is correct for planar $2$-edge-connected graphs. However, we did not understand how the occurrence of $v$ in faces other than $f$ can have an impact on $\pi(f)$ besides the one created via the components of $G-v$.}
in $\pi (f)$ as many times as the number of connected components of $G-v$.
We call this number the {\emph{multiplicity}} of $v$ in $f$ and we denote it by $\mu _f(v)$.

The \emph{radial graph} of $(\Gamma, U)$ is the multigraph $R_{G}$ that has $V(G)\cup F(\Gamma, U)$ as its vertex set and that contains as edges every pair $vf$, where $v$ is a vertex incident to the face $f$, with multiplicity $\mu _f(v)$. 
Notice that $R_{G}$ has a spherical embedding $(\Gamma _R,U_R)$ where 
\begin{itemize}
\item    each point of $U_R$ corresponding to a vertex $v\in V(G)$ is the point of $U$ 
corresponding to $v$,  
\item  each point of $U_R$ corresponding to a face $f$ of $(\Gamma, U)$ is a point of $f$, and  
\item each arc between a point $x$ corresponding to a vertex $v\in V(G)$ and a point $y$ corresponding to a face $f\in F(\Gamma, U)$ is a subset of $f$.
Moreover there are $\mu _f(v)$ such arcs.
\item  if $a_{1},\ldots,a_{p},a_{1},$ are the arcs with $y$ as a common endpoint where $y$ is the point corresponding to face $f$, arranged following the cyclic ordering that they appear around $y$, and $x_{1},\ldots,x_{p},x_{1}$ are the other endpoints of these arcs, then $x_{1},\ldots,x_{p},x_{1}$ is the same cyclic ordering  as $\pi (f)$.
\end{itemize}

We stress that the notion of a radial graph can be defined also without the $2$-edge connectivity requirement.
However, we prefer to avoid this because the definition becomes more technical and requires the use of loops, that we prefer to avoid.
Also, the $2$-edge connectivity implies that  $R_{G}$ is a simple graph and that  all faces of $(\Gamma _R,U_R)$ are incident to {exactly \textsl{four}}%
\revthree{18. page 12, line 42: four ~$\longrightarrow$~ two?} 
\ans{In an $2$-edge connected radial graph, all cases are ``squares''.}
vertices of $R_G$ that are also vertices of $G$.
Moreover, the edges of $G$ bijectively correspond the faces of $R_{G}$.

\begin{proposition}[\cite{SeymourT94callr}]
\label{llkseferatcatc}
Let $k$ be a positive integer and let $(\Gamma, U)$ be a spherical embedding of some $2$-edge connected graph $G$.
Then $G$ has a tangle of order $k$ if and only if $(\Gamma _R,U_R)$ has a slope of order $k$.
\end{proposition}

At this point we wish to stress that, given an almost spherical embedding $(G',\Gamma ,U,\Bcal,\Pcal)$ of a graph $G$, an argument similar to the one in the last paragraph of the proof of \cref{ope_tyop} already implies a ``multiplicative drop'' from the branchwidth of $G$ to the branchwidth of $G'$ and this is  certainly possible even if $G'$ is embedded to some surface of higher genus. 
However, what we pursue in this paper is  an ``additive drop'' for the case where $G'$ is planar.
As we see, this is a very particular property of branchwidth that is based on the exact min-max duality of \cref{llkseferatcatc}.
This is because the tangle-duality is tight only for branchwidth and because,  in planar graphs, for every minimal {separator $S$} of a spherically embedded graph  $G,$ there is a noose $N$ where $V_{G}(N)=S$ (see \cite{SeymourT94callr}).

\begin{theorem}
\label{wthiskslopes}
Let $G$ be a graph and let~$(G',\Gamma ,U,\Bcal,\Pcal)$ be an almost spherical embedding  of~$G$ of breadth $b$ and width at most $w$.
Assume also that $G'$ is 2-edge connected. Then $\bw(G)-\bw(G')\leq 2b\cdot w$.
\end{theorem}

\begin{proof}
Let $(G',\Gamma ,U,\Bcal,\Pcal)$ be an almost spherical embedding  of~$G$ as in \cref{defonfoe}.
As $G'$ is $2$-connected it contains a cycle, therefore $\bw(G)\geq \bw(G')\geq 2$.  
Let $R_{G'}$ be the {radial graph} of $(\Gamma, U)$ and let $(\Gamma _R,U_R)$ be a spherical embedding of $R_{G'}$.
Because of \cref{llkseferatcatc} and \cref{minmax_tangle}, it is enough to prove that 
if $G$ has a tangle of order $k\coloneqq\bw(G),$  then $(\Gamma _R,U_R)$ has a slope of order $k'\coloneqq k-2b\cdot w$.

The $2$-connectivity of $G'$ implies that all faces of $(\Gamma _R,U_R)$ are incident to exactly four vertices and they are bijectively mapped to the edges of $G'$.
Notice that every cycle $C$ of $R_{G'}$ can be seen as a noose of $(\Gamma, U)$ where $V(C)$ contains exactly $|V(C)|/2$ vertices from $G'$.
We denote these vertices by $V_{C}$. 

For every cycle $C$ of $R_{G'}$ of length $<2k'$, we consider the separation $(A_{C},B_{C})$ of $G'$ where $A_{C}$ contains the \rred{$<k'$} vertices \rred{of $G'$} inside one of the closed disks of $\Sbbb^2$ bounded by $C$ in  $(\Gamma, U)$ and $B_{C}$ contains the \rred{$<k'$} vertices   \rred{of $G'$}  that are inside the other \rred{(here we use the fact  that in each cycle of $R_{G'}$ the vertices of $G'$ and the vertices corresponding to faces of $G'$ are alternating, therefore they are of equal size)}.
Clearly $V_{C}=A_{C}\cap B_{V},$ therefore both $(A_{C},B_{C})$ and $(B_{C},A_{C})$ are {separations} \rred{of $G'$} of order $<k'$.
\revone{16. Page 13, line 21: Please state clearly in which graph(s) you consider the separations, and why the parameter halves (from $2k'$ to $k'$).}
\ans{We now provide explanation based on the alternation of vertices and faces in the cycles of the radial graph.}

Given a cycle $C$ of $R_{G'}$ of length $<2k'$ and the separation $(A_{C},B_{C})$ we define $\tilde{A}_{C}=\bigcup_{v\in A_{C}}V_{v}$ and $\tilde{B}_{C}=\bigcup_{v\in B_{C}}V_{v}$ (see \eqref{dfkol_dm} for the definition of $V_{v}$).
As, in the spherical embedding $(\Gamma, U)$ of $G',$ the curve $C$ meets each of the faces of $G'$ once, we have that, for every $i\in[r]$, no more than two vertices of $V(N_i)$ are met by $C$.
Therefore, $V_{C}$ contains at most $2b$ rich vertices and this in turn implies that   
$\tilde{A}_{C}\cap \tilde{B}_{C}$ contains at most $|V_{C}|-2b+2b\cdot (w+1)<k'+2b\cdot w$ vertices.
We conclude that both $(\tilde{A}_{C}, \tilde{B}_{C})$ and $(\tilde{B}_{C}, \tilde{A}_{C})$ are separations of $G$ of order $<k'+2b\cdot w=k$.

Let $\mathcal{T}$ be a tangle of $G$ of order $k$.
We use $\mathcal{T}$ in to build a slope $\mathsf{ins}$ of $G'$ of order $k$ as follows.
Let $C$ be a cycle of $R_{G'}$ of length $<2k'$.
If $(\tilde{A}_{C}, \tilde{B}_{C})\in \mathcal{T}$ then we set $\mathsf{ins}(C)$ to be the closed disk bounded by $C$ that contains $A_{C}$ and if $(\tilde{B}_{C}, \tilde{A}_{C})\in \mathcal{T},$ then  $\mathsf{ins}(C)$ is the closed disk bounded by $C$ that contains $B_{C}$.

We claim that $\mathsf{ins}$ is a uniform slope of $R_{G'}$ of order $k'$. For this we have to prove that S1, S2, and S3 are satisfied.
For this we give first some definitions. Given a cycle $C$ of length $<k$ of $R_{G'},$ we define $A_{C}=\mathsf{ins}(C)\cap V(G'),$ $B_{C}=\overline{\Sbbb_{0} \setminus \mathsf{ins}(C)}\cap V(G')$ and observe that $(A_{C},B_{C})$ is a separation of $G'$ of order $<k'$ which corresponds to the separation $(\tilde{A}_{C},\tilde{B}_{C})$ of  order $<k$ in $\mathcal{T}$.\medskip

For S1, assume to the contrary that $C$ and $C'$ are cycles of $R_{G'}$ of length $<k$
where $C\subseteq \mathsf{ins}(C')$ and $\mathsf{ins}(C)\subseteq \overline{\Sbbb_{0} \setminus \mathsf{ins}(C')}$.
This implies that $\mathsf{ins}(C)\cup\mathsf{ins}(C')=\Sbbb_{0} $.
This implies that $G'[A_{C}]\cup G'[A_{C'}]=G'$ which in turn implies that  $G[\tilde{A}_{C}]\cup G[\tilde{A}_{C'}]=G$.
But then, then the triple $(\tilde{A}_{C},\tilde{B}_{C}), (\tilde{A}_{C},\tilde{B}_{C}), (\tilde{A}_{C'},\tilde{B}_{C'})$ violates condition~\eqref{tangleax} for $\mathcal{T}$.\medskip
 
For S2, assume to the contrary that $P_1,P_2,P_3$ are three paths of $R_{G'}$ joining two distinct vertices $x$ and $y$ and such that the three cycles $C_{1,2}\coloneqq P_{1}\cup P_{2},$ $C_{2,3}\coloneqq P_{2}\cup P_{3},$ and $C_{3,1}\coloneqq P_{3}\cup P_{1}$ have length $<k,$ and  $\mathsf{ins}(C_{1,2})\cup\mathsf{ins}(C_{2,3})\cup\mathsf{ins}(C_{3,1})= \Sbbb_{0} $.
This implies that $G'[A_{C_{1,2}}]\cup G'[A_{C_{2,3}}]\cup G'[A_{C_{3,1}}]=G',$ therefore $G[\tilde{A}_{C_{1,2}}]\cup G[\tilde{A}_{C_{2,3}}]\cup G[\tilde{A}_{C_{3,1}}]=G$.
But then the triple $(\tilde{A}_{C_{1,2}},\tilde{B}_{C_{1,2}}), (\tilde{A}_{C_{2,3}},\tilde{B}_{C_{2,3}}), (\tilde{A}_{C_{3,1}},\tilde{B}_{C_{3,1}})$ violates  condition~\eqref{tangleax} for $\mathcal{T}$.\medskip

For S3, assume to the contrary that  $\mathsf{ins}$ is not uniform, i.e.,  there is a (square) face $f$ of $R_{G'}$  that is not a subset of $\mathsf{ins}(C)$ for some cycle of $R_{G}$ of length $<k$.
Then this is also the case for the cycle $C_{f}$ of length $4$ corresponding to the boundary of $f$.
This implies that $\mathsf{ins}(C_f)=\Sbbb_{0}\setminus f,$ which in turn implies that $A_{C_{f}}=V(G')$ and therefore $\tilde{A}_{C_{f}}=V(G)$.
But then the triple  $(\tilde{A}_{C_f},\tilde{B}_{C_{f}}),$  $(\tilde{A}_{C_f},\tilde{B}_{C_{f}}),$  $(\tilde{A}_{C_f},\tilde{B}_{C_{f}})$ violates condition~\eqref{tangleax} for $\mathcal{T}$.
\end{proof}

\rred{
Given a graph $G$ and an edge set $E\subseteq E(G)$,
we denote by $G+E$ the graph obtained by $G$ if we add all the edges in $E$ in $G$ by
agreeing that for every edge $e\in E(G)\cap E$,  the multiplicity  of $e$ in $G+E$ is increased by one.
 }

\rred{
\begin{lemma}
\label{frt6y}
Let $(G',\Gamma ,U,\Bcal,\Pcal)$  be an almost spherical embedding of a graph $G$, with $\bw(G')\geq 2$.
There is a set $E\subseteq {V(G')\choose 2}$ such that, if $G^{\mathsf{new}}\coloneqq G+E$ and  $G^{\prime,\mathsf{new}}\coloneqq G'+E$   then $G^{\prime,\mathsf{new}}$ is $2$-edge-connected and $\bw(G^{\prime,\mathsf{new}})=\bw(G')$.
Moreover, ${G}^{\mathsf{new}}$ has a an almost spherical embedding $({G}^{\prime,\mathsf{new}},\Gamma ^{\mathsf{new}},U,\Bcal,\Pcal)$.
\end{lemma}}%
\revone{17. Page 14, line 1+: I suggest to break the long sentence into a more understandable description.}
\ans{Done. We added the definition of $G+E$ before the statement of the lemma.}

\begin{proof}
Let $b\coloneqq |\Bcal|$.
If $G'$ contains some bridge $e$, then we duplicate it and do the same in $G$. Notice that 
the new edge can be drawn as a new arc in the current $\Gamma $ \change{so to avoid}{in a way that avoids}%
\revthree{19.  page 14, line 9: ``so to avoid'' looks like a typo.}
\ans{Fixed.}
 the disks of the vortices of $\Bcal$. 
 {This makes all connected components}%
\revone{18. Page 14, line 10+: I find the beginning of the proof overcomplicated.} 
\ans{The complication comes because we wish to stress that the new edges should not be drawn in the vortices.}
of $G'$
$2$-edge-connected.
Moreover, as the branchwidth of $G'$ is at least two, this duplication \change{maintains our branchwidth invariant}{does not change the branchwidth of $G$ or $G'$}.
We also update $(\Gamma, U)$ accordingly. 

Our next step is to make $G'$ connected.
For this consider two connected components $C_{1}$ and $C_{2}$ of $G'$ and observe that $\Sbbb_{0}\setminus\bigcup_{i\in[b]}\Delta_{i}$ contains an open disk $\Delta$  whose closure contains a point of $U$ corresponding to a vertex $v_{1}$ of $C_{1}$ and a point of $U$ corresponding to a vertex $v_{2}$ of $C_{2}$ (recall that $\Delta_{1},\ldots,\Delta_{b}$ are the disks of the vortices in $\Bcal$).
Then add to $G'$ (and also to $G$) the double edge $v_1v_2$ and update $(\Gamma, U)$ by drawing the corresponding two new arcs inside $\Delta$ so that their endpoints are 
the points of $U$ corresponding to $v_{1}$ and $v_{2}$.
Using \cref{alltog}, for $r=1$, it  follows that the addition of this double edge \change{keeps the branchwidth  invariant}{does not change the branchwidth}%
\revone{19. Page 14, line 22: ``keeps the branchwidth invariant'' ? Do you simply mean that the $\bw$. does not change?}
\ans{Indeed. Fixed.}
 of $G'$.
Notice that the new $G'$ has less connected components than the old one. 

By adding double edges in $G'$ and $G$ as above, until $G'$ becomes connected 
(and therefore also $2$-edge-connected), we end up with two graphs \rred{that we denote by} $G^{\mathsf{new}}$ and $G^{\prime,\mathsf{new}}$ \rred{respectively} and an almost spherical embedding {$({G}_{t}^{\prime,\mathsf{new}},\Gamma ^{\mathsf{new}},U,\Bcal,\Pcal)$ of ${G}_{t}^{\mathsf{new}},$}%
\revthree{20. page 14, line 44: ``$^{\mathsf{new}}$'' is not an operator, so $\overline{G}^{\mathsf{new}}$ should be defined.}
\ans{Fixed.}
 as required.
\end{proof}

\begin{proof}[Proof of \cref{maintscomb}]
We consider the decomposition $(T,\beta)$ of \cref{from_surfex} and let $q\coloneqq f_{\ref{from_surfex}}(k)$.

Let $t\in V(T)$.
If $|V(G_{t})|\leq q$, we set $X_{t}=V(G_{t})$ and conditions i)--iv) hold trivially as $G_{t}-X_{t}$ is the empty graph.\medskip

Let now $|V(G_{t})>q$ and let $(G_t',\Gamma _t,U_{t},\Bcal_t,\Pcal_t)$ be an almost spherical embedding of $\overline{G}_t\coloneqq G_t-A_{t}$ of width $\leq q$ and breadth $\leq q$ and for some $A_{t}\subseteq V(G_{t})$ where $|A_{t}|\leq q$.

Let $Z_{t}=(\cupall \Bcal_t)\setminus V(G'_{t})$ and $X_{t}=A_{t}\cup Z_{t}$, i.e.,
$X_{t}$ contains all apices and all the internal vertices of  the vortices in $\Bcal_t$. 
Clearly, $G_{t}-X_{t}=\overline{G}_{t}-Z_{i}=G'_{t}$.
Condition i) follows as $G'_{t}$ has a spherical embedding. 
Conditions ii) and iii) follow from B1 and B2 respectively.

For Condition iv),  recall first that B3 permits us 
to assume that $\bw(G'_{t})\geq 2$ (\cref{remksiop}). 
Therefore we may apply \cref{frt6y}
and obtain a $2$-edge-connected graph $G_t^{\prime,\mathsf{new}}$ where $\bw(G_{t}')=\bw({G}_{t}^{\prime,\mathsf{new}})$ and an almost spherical embedding $({G}_{t}^{\prime,\mathsf{new}},\Gamma _t^{\mathsf{new}},U_{t},\Bcal_t,\Pcal_t)$ of $\overline{G}_{t}^{\mathsf{new}}$.  As $\overline{G}_{t}$ is a spanning subgraph of 
$\overline{G}_{t}^{\mathsf{new}}$, we also have that $\bw(\overline{G}_{t})\leq \bw(\overline{G}_{t}^{\mathsf{new}})$.

We may now apply \cref{wthiskslopes} on  $({G}_{t}^{\prime,\mathsf{new}},\Gamma _t^{\mathsf{new}},U_{t},\Bcal_t,\Pcal_t)$ 
and derive that $\bw(\overline{G}_t^{\mathsf{new}})\leq 2q^3+\bw({G}_{t}^{\prime,\mathsf{new}})$. This implies that 
$\bw(\overline{G}_t)\leq 2q^3+\bw({G}_{t}^{\prime})$.
Therefore, we obtain $\bw(G_{t})  \leq  |A_{t}|+
\bw(\overline{G}_t)\leq q+2q^3+\bw({G}_{t}^{\prime})$ from \cref{apex_bw}, and thus, iv) also holds by setting $f_{\ref{maintscomb}}(k)\coloneqq q+2q^3$.
\end{proof}

\section{Conclusion}
\label{opne_all}

In this paper we show that for classes of graphs not containing as a subgraph a handle wall or a crosscap wall, there exists an constant-additive approximation algorithm  for branchwidth (\cref{mainth}).
Our algorithm is based on the algorithms in \cref{pr_bodl_linear} and \cref{pr_rat_k}
that can also {construct}%
\revone{20. Page 14, line 58: I do not find this formulation correct, since your wording of \cref{pr_bodl_linear} and \cref{pr_bodl_linear} does not include a construction. Can you be more precise and cite your sources? Or, cannot you use a self-reduction approach to construct a decomposition?} 
\ans{The result of Bodlaender in \cite{BodlaenderT97const} can also construct the tree decomposition. This is visible in the full version of the paper, \cite{BodlaenderT97const}, that we now site.}
the branch-decompositions in question 

\change{(see \mbox{\cite{GuT08optim}} for the constructive version of  \mbox{\cref{pr_rat_k}}}{See Bodlaender and Thilikos \cite{BodlaenderT97const} as well as Gu and Tamaki \cite{GuT08optim} for the constructive versions of 
\cref{pr_bodl_linear} and \cref{pr_bodl_linear} respectively.}

This implies that our algorithm can also be enhanced to output the corresponding branch decomposition.
Another issue is whether the cubic running time can be improved. This 
depends on two factors. The first is the cubic time required to construct the decomposition in \cref{from_surfex}.
This running time might be improved to a quadratic one using the technique of  Grohe, Kawarabayashi, and  Reed in \cite{GroheKR13}.
The second is the ``Ratcatcher'' algorithm in \cref{pr_rat_k}
that requires cubic time and, to our knowledge, no faster algorithm exists for 
computing the branchwidth of a planar graph.                  
\smallskip

As we already mentioned, the exact polynomial computation of branchwidth is 
known to be possible for graph classes excluding a singly-crossing graph as a minor \cite{DemaineHNRT04appro}.
As singly-crossing graphs are minors of both the handle wall and the 
crosscap wall, our result can be seen as an ``approximation'' extension of the results in \cite{DemaineHNRT04appro}. However the question on whether (and when) an exact polynomial algorithm exists remains open. Our guess is the following.

\begin{conjecture}
If $H$ is a graph that can be embedded both in the torus and in the projective plane, then branchwidth can be computed in polynomial time in $H$-minor free graphs.
\end{conjecture}

Notice that singly-crossing graphs are minors of both handle walls or a crosscap walls. 
However, the \change{inverse}{converse}%
\revone{21. Page 15, line 22: inverse ~$\longrightarrow$~ converse}
\ans{Fixed.}
 is not the case.
For instance, $K_{6}$ can be embedded both in the torus and in the projective plane, but it is not a singly-crossing graph.
So a first step towards pursuing a proof of the above conjecture would be to look for a polynomial time algorithm for branchwidth on $K_{6}$-minor free graphs or proving that this is not possible subject to some complexity theory assumption.


\end{document}